\pdfoutput=1 

\documentclass[11pt]{amsart}

\usepackage{amssymb,amsthm,enumitem,colonequals,tikz-cd,microtype}
\usepackage[normalem]{ulem} 

\usepackage[osf]{Baskervaldx}
\usepackage[baskervaldx]{newtxmath}
\usepackage[cal=boondoxo, bb=px]{mathalfa} 

\usepackage[top=3.75cm, bottom=3cm, left=3.5cm, right=3.5cm]{geometry}

\usepackage{xcolor}
\colorlet{darkblue}{blue!55!black}
\colorlet{darkcyan}{cyan!50!black}
\colorlet{darkgreen}{green!60!black}

\PassOptionsToPackage{hyphens}{url}
\usepackage{hyperref}
\hypersetup{
    colorlinks=true,
    linkcolor=darkblue,
    urlcolor=darkcyan,
    citecolor=darkgreen,
}

\def\eqref#1{\textcolor{darkblue}{(\ref{#1})}}

\usepackage[nameinlink]{cleveref} 
\Crefformat{section}{#2\S#1#3}
\Crefmultiformat{section}{#2\S\S#1#3}{ and~#2#1#3}{, #2#1#3}{, and~#2#1#3}
\Crefname{equation}{Diagram}{Diagrams}
\crefname{equation}{Equation}{Equations}

\usepackage[pagewise]{lineno}
\let\oldequation\equation
\let\oldendequation\endequation
\renewenvironment{equation}{\linenomathNonumbers\oldequation}{\oldendequation\endlinenomath}
\expandafter\let\expandafter\oldequationstar\csname equation*\endcsname
\expandafter\let\expandafter\oldendequationstar\csname endequation*\endcsname
\renewenvironment{equation*}{\linenomathNonumbers\oldequationstar}{\oldendequationstar\endlinenomath}
\let\oldalign\align
\let\oldendalign\endalign
\renewenvironment{align}{\linenomathNonumbers\oldalign}{\oldendalign\endlinenomath}
\expandafter\let\expandafter\oldalignstar\csname align*\endcsname
\expandafter\let\expandafter\oldendalignstar\csname endalign*\endcsname
\renewenvironment{align*}{\linenomathNonumbers\oldalignstar}{\oldendalignstar\endlinenomath}

\theoremstyle{plain}
\newtheorem{theorem}{Theorem}[section]
\newtheorem{lemma}[theorem]{Lemma}
\newtheorem{corollary}[theorem]{Corollary}
\newtheorem{proposition}[theorem]{Proposition}

\theoremstyle{definition}
\newtheorem{definition}[theorem]{Definition}
\newtheorem{example}[theorem]{Example}
\newtheorem{remark}[theorem]{Remark}

\newtheorem*{ack}{Acknowledgments}

\AddToHook{env/conjecture/begin}{\crefalias{theorem}{conjecture}}
\AddToHook{env/lemma/begin}{\crefalias{theorem}{lemma}}
\AddToHook{env/corollary/begin}{\crefalias{theorem}{corollary}}
\AddToHook{env/proposition/begin}{\crefalias{theorem}{proposition}}
\AddToHook{env/definition/begin}{\crefalias{theorem}{definition}}
\AddToHook{env/remark/begin}{\crefalias{theorem}{remark}}
\AddToHook{env/setup/begin}{\crefalias{theorem}{setup}}
\AddToHook{env/example/begin}{\crefalias{theorem}{example}}
\AddToHook{env/conjecture/begin}{\crefalias{theorem}{Conjecture}}
\AddToHook{env/hypothesis/begin}{\crefalias{theorem}{Hypothesis}}
\AddToHook{env/reminder/begin}{\crefalias{theorem}{Reminder}}

\numberwithin{equation}{section}
\numberwithin{theorem}{section}

\title{Categorical characterizations of regularity for algebraic stacks}

\author[T.~De Deyn]{Timothy De Deyn}
\address{T.~De Deyn,
Max Planck Institute for Mathematics,
Bonn, Germany}
\email{dedeyn@mpim-bonn.mpg.de}

\author[P.~Lank]{Pat Lank}
\address{P.~Lank,
Dipartimento di Matematica “F. Enriques”, Universit\`{a} degli Studi di Milano, Via Cesare
Saldini 50, 20133 Milano, Italy}
\email{plankmathematics@gmail.com}

\author[K.~Manali Rahul]{Kabeer Manali Rahul}
\address{K.~Manali Rahul,
Max Planck Institute for Mathematics,
Bonn, Germany}
\email{kabeermr.maths@gmail.com}

\author[F. ~Peng]{Fei Peng}
\address{F.~Peng,
School of Mathematics \& Statistics\\
The University of Melbourne\\
Parkville, VIC, 3010\\
Australia}
\email{pengf2@student.unimelb.edu.au}

\date{\today}



\begin{document}
    
\begin{abstract}
    For a Noetherian scheme $X$ of finite Krull dimension, Neeman recently established two characterizations of the regularity of $X$ using strong generators and bounded $t$-structures on $\operatorname{Perf}(X)$. In this note, we obtain variants of Neeman's results for large classes of Noetherian algebraic stacks. An important intermediate step is the fact that $X$ is regular if and only if $\operatorname{Perf}(X)=D_{\operatorname{coh}}^b(X)$, which we establish for Noetherian algebraic stacks. Our approach also yields a criterion for the existence of classical generators for the bounded derived categories of coherent sheaves on algebraic stacks, generalizing previous results for commutative rings and schemes.
\end{abstract}

\maketitle

\tableofcontents

\section{Introduction}
\label{sec:intro}

Let $X$ be a Noetherian scheme and denote by $D_{\operatorname{qc}}(X)$ the unbounded derived category of $\mathcal{O}_X$-modules with quasi-coherent cohomology sheaves. It contains two subcategories that are closely related to the geometry of $X$, namely the category of perfect complexes $\operatorname{Perf}(X)$ and the bounded derived category of coherent sheaves $D^b_{\operatorname{coh}}(X)$ on $X$. In particular, it is known that properties of $\operatorname{Perf}(X)$ as a triangulated category can be used to characterize the regularity of $X$.
The main goal of this article is to extend these categorical characterizations of regularity from schemes to algebraic stacks.

Recall that a triangulated category $\mathcal{T}$ admits a \textit{strong generator} if there exists an object $G \in \mathcal{T}$ such that every object of $\mathcal{T}$ can be obtained from $G$ using finite coproducts, direct summands, shifts, and at most $n$ cones for some fixed integer $n$. It is shown in \cite[Theorem 3.1.4]{Bondal/VandenBergh:2003} that the category of perfect complexes on a smooth variety over a field admits a strong generator. This result was significantly generalized in \cite[Theorem 0.5]{Neeman:2021a} where it is proved that a separated Noetherian scheme $X$ of finite Krull dimension is regular if, and only if, $\operatorname{Perf}(X)$ admits a strong generator. An algebraic stack is said to be \textit{concentrated} if perfect complexes coincide with the compact objects (see \Cref{sec:prelim_algebraic_stacks_compacts}). 

Our first main result is a generalization of \cite[Theorem 0.5]{Neeman:2021a} to concentrated algebraic stacks.

\begin{theorem}
    \label{introthm:stacky_bondal_vandenbergh}
    Let $\mathcal{X}$ be a concentrated separated Noetherian algebraic stack with affine diagonal and of finite Krull dimension. 
    Then $\mathcal{X}$ is regular if, and only if, $\operatorname{Perf} (\mathcal{X})$ admits a strong generator.
\end{theorem}

\Cref{introthm:stacky_bondal_vandenbergh} is a consequence of a more general statement \Cref{thm:stacky_bondal_vandenbergh}. It is possible to relax the separated assumption in \Cref{introthm:stacky_bondal_vandenbergh} following \cite{Jatoba:2021}. We do not treat this in this article. The Noetherianity assumption, however, is necessary even in the scheme case. See \cite{Stevenson:2025} for a counterexample. Moreover, the existence of strong generators implies the finiteness of the Krull dimension in the Noetherian case. Therefore, the finiteness of the Krull dimension cannot be relaxed either. The concentrated assumption is also necessary. It is automatically satisfied for quasi-compact quasi-separated schemes, but this is no longer true in general for algebraic stacks (see e.g.\ \Cref{ex:bkga_concentrated_perf_fails}).

On the other hand, a result of Antieau--Gepner--Heller \cite{Antieau/Gepner/Heller:2019} identifies obstructions to the existence of bounded $t$-structures on stable $\infty$-categories in terms of negative $K$-theory. Furthermore, loc.\ cit.\ conjectured that a Noetherian scheme $X$ is regular if, and only if, $\operatorname{Perf}(X)$ admits a bounded $t$-structure. This conjecture was recently resolved by Neeman for Noetherian schemes of finite Krull dimension \cite{Neeman:2023}. Moreover, loc.\ cit.\ proves a relative form of this result: for a closed subset $Z \subseteq X$, writing $\operatorname{Perf}_Z(X)$ for the category of perfect complexes supported on $Z$, one has that $Z \subseteq \operatorname{reg}(X)$ if, and only if, $\operatorname{Perf}_Z(X)$ admits a bounded $t$-structure.

Our next result extends Neeman's theorem to suitably well-behaved algebraic stacks.

\begin{theorem}
    \label{introthm:stacky_neeman_bounded_t_structure}
    Let $\mathcal{X}$ be a concentrated Noetherian algebraic stack of finite Krull dimension that either    
    \begin{enumerate}
        \item has quasi-finite and separated diagonal or
        \item is a Deligne--Mumford stack of characteristic zero.
    \end{enumerate}
    Then for any closed $Z\subseteq |\mathcal{X}|$ one has $Z\subseteq\operatorname{reg}(\mathcal{X})$ if, and only if, $\operatorname{Perf}_Z (\mathcal{X})$ admits a bounded $t$-structure.
\end{theorem}

This appears later as \Cref{thm:stacky_neeman_bounded_t_structure}. See \Cref{def:regular_locus} for the definition of $\operatorname{reg}(\mathcal{X})$. While the proof of \Cref{introthm:stacky_neeman_bounded_t_structure} closely follows that of \cite{Neeman:2023}, substantial additional work is required to adapt it to the stack theoretic setting. The main technical obstruction is to show that, for $Z$ the complement of a quasi-compact open, the standard $t$-structure on $D_{\operatorname{qc},Z}(\mathcal{X})$ lies in the preferred equivalence class (see \Cref{prop:preferred_equivalence_classes}). Moreover, our argument also relies on the induction principle for algebraic stacks \cite{Hall/Rydh:2018}. In particular, our approach avoids the need to establish weak approximability (in the sense of \cite{Neeman:2022,Neeman:2023}) for $D_{\operatorname{qc},Z}(\mathcal{X})$, which may be of independent interest. We also wish to remark that \Cref{prop:preferred_equivalence_classes} fails miserably for stacks with infinite stabilizers such as $B\mathbb{G}_m$. This follows from the absence of the preferred equivalence class of $t$-structures generated by a countable family of compact objects.

In the scheme case, both characterizations mentioned earlier ultimately rely on the elementary but important fact that a Noetherian scheme is regular if, and only if, $\operatorname{Perf}(X)=D_{\operatorname{coh}}^{b}(X)$. We obtain the following variant of this classical fact for algebraic stacks as part of our argument for our main theorems. 

\begin{proposition}
    \label{intrcor:derived_characterization_regularity_stacks}
    Let $\mathcal{X}$ be a decent quasi-compact locally Noetherian algebraic stack. Then $\mathcal{X}$ is regular if, and only if, $\operatorname{Perf}(\mathcal{X})= D^b_{\operatorname{coh}}(\mathcal{X})$.
\end{proposition}

\Cref{intrcor:derived_characterization_regularity_stacks} is a consequence of a relative version \Cref{prop:closed_subset_regular_locus}. We expect \Cref{intrcor:derived_characterization_regularity_stacks} to be folklore among experts, although we were unable to find a reference in the literature. The forward implication is immediate, while the converse is more subtle. One issue is that perfect complexes on an algebraic stack $\mathcal{X}$ do not always coincide with the compact objects of $D_{\operatorname{qc}}(\mathcal{X})$. Thus, for nonconcentrated algebraic stacks, a formulation purely in terms of compact objects is not possible.

\begin{example}
    \label{ex:bkga_concentrated_perf_fails}
    Consider the classifying stack $B_{k}\mathbb{G}_a$ over a field $k$ of characteristic $p>0$. The stack $B_{k}\mathbb{G}_a$ is smooth and of finite type over $k$, and the category of quasi-coherent sheaves on $B_{k}\mathbb{G}_a$ is equivalent to the category of $\mathbb{G}_a$-representations over $k$. By \cite[Proposition\ 3.1]{Hall/Neeman/Rydh:2019}, we known that $D_{\operatorname{qc}}(B_{k}\mathbb{G}_a)$ has no nonzero compact objects. However, we see that $\operatorname{Perf}(B_{k}\mathbb{G}_a)$ and $D^b_{\operatorname{coh}}(B_{k}\mathbb{G}_a)$ are both equivalent to the derived categories of finite dimensional $\mathbb{G}_a$-representations, which is nontrivial. Nevertheless, \Cref{intrcor:derived_characterization_regularity_stacks} still holds for $B_{k}\mathbb{G}_a$, but one must not use compact objects to characterize regularity.
    For the case of finite group quotients see \Cref{ex:C2}.
\end{example}

To circumvent the concentratedness hypothesis, we begin by detecting regularity at closed points via the perfectness of the structure sheaves of their associated residual gerbes. At this stage, detecting regularity at arbitrary points is unclear. However, the hypotheses of \Cref{intrcor:derived_characterization_regularity_stacks} ensure both the existence of sufficiently many closed points and that the regular locus is stable under generalization. These properties allow us to deduce regularity at all points of the algebraic stack.

A natural consequence of this approach is the following stack theoretic analogue of the existence of a  classical generator for the bounded derived category of coherent sheaves.

\begin{proposition}
    \label{introprop:j_condition_stacks}
    Let $\mathcal{X}$ be a concentrated Noetherian algebraic stack satisfying the $1$-Thomason condition. 
    Then the following are equivalent:
    \begin{enumerate}
        \item $\operatorname{reg}(\mathcal{Z})$ contains a nonempty open for every integral closed substack $\mathcal{Z}$ of $\mathcal{X}$
        \item $\operatorname{reg}(\mathcal{Z})$ is open for every integral closed substack $\mathcal{Z}$ of $\mathcal{X}$
        \item $D^b_{\operatorname{coh}}(\mathcal{Z})$ admits a classical generator for every integral closed substack $\mathcal{Z}$ of $\mathcal{X}$.
    \end{enumerate}
    Moreover, if any of these conditions hold, then  $D^b_{\operatorname{coh}}(\mathcal{Y})$ admits a classical generator for any closed substack $\mathcal{Y}$ of $\mathcal{X}$.
\end{proposition}

This result appears later as \Cref{prop:j_condition_stacks}. It generalizes previous results of Iyengar and Takahashi for commutative rings \cite{Iyengar/Takahashi:2019} and of Dey and the second author for schemes \cite{Dey/Lank:2024a}. See \Cref{sec:prelim_algebraic_stacks_thomason} for the definition of the Thomason condition. Many Noetherian algebraic stacks satisfy the $1$-Thomason condition. This includes, for instance, those with quasi-finite separated diagonals \cite[Theorem A]{Hall/Rydh:2017}. Also, concentrated regular algebraic stacks with quasi-finite diagonal \cite[Corollary 1.2]{Lank:2026}.

\begin{ack}
    Timothy De Deyn was supported by ERC Consolidator Grant 101001227 (MMiMMa). Pat Lank and Kabeer Manali Rahul were supported by ERC Advanced Grant 101095900-TriCatApp. Kabeer Manali Rahul was also supported by an Australian Government Research Training Program Scholarship. Fei Peng was supported by the Australian Research Council DP210103397 and FT210100405, a Melbourne Research Scholarship, and the Science Abroad Travelling Scholarship offered by the University of Melbourne. The authors thank Jack Hall for discussions and suggestions.
\end{ack}

\section{Preliminaries}
\label{sec:prelim}

\subsection{Triangulated categories}
\label{sec:prelim_tricats}

Let $\mathcal{T}$ be a triangulated category with shift functor $[1]$.

\subsubsection{Generation}
\label{sec:prelim_generation}

We briefly discuss generation of triangulated categories. See \cite{Bondal/VandenBergh:2003, Rouquier:2008,Neeman:2021a} for detail. Fix a subcategory $\mathcal{S}\subseteq\mathcal{T}$. Define $\operatorname{add}(\mathcal{S})$ to be the smallest strictly full subcategory of $\mathcal{T}$ containing $\mathcal{S}$ closed under shifts, finite coproducts, and direct summands. Inductively, define 
\begin{displaymath}
    \langle\mathcal{S}\rangle_n :=
    \begin{cases}
        \operatorname{add}(\varnothing) & n=0, \\
        \operatorname{add}(\mathcal{S}) & n=1, \\
        \operatorname{add}(\{ \operatorname{cone}\phi \mid \phi \in \operatorname{Hom}(\langle \mathcal{S} \rangle_{n-1}, \langle \mathcal{S} \rangle_1) \}) & n>1.
    \end{cases}
\end{displaymath}
An object $G\in \mathcal{T}$ is a called a \textbf{classical generator} if $\mathcal{T} = \langle G\rangle$ and a \textbf{strong generator} if $\mathcal{T} = \langle G\rangle_{n+1}$ for some $n\geq 0$. We say that triangulated subcategory is called \textbf{thick} if it is closed under direct summands.
In fact, $\langle\mathcal{S}\rangle := \cup_{n\geq 0} \langle\mathcal{S}\rangle_n$ is the smallest thick subcategory of $\mathcal{T}$ containing $\mathcal{S}$.

If $\mathcal{T}$ admits small coproducts, define $\operatorname{Add}(\mathcal{S})$ to be the smallest strictly full subcategory of $\mathcal{T}$ containing $\mathcal{S}$ closed under shifts, small coproducts, and direct summands. 
Similarly, inductively define 
\begin{displaymath}
    \overline{\langle\mathcal{S}\rangle}_n :=
    \begin{cases}
        \operatorname{Add}(\varnothing) & n=0, \\
        \operatorname{Add}(\mathcal{S}) & n=1, \\
        \operatorname{Add}(\{ \operatorname{cone}\phi \mid \phi \in \operatorname{Hom}(\overline{\langle \mathcal{S} \rangle}_{n-1}, \overline{\langle \mathcal{S} \rangle}_1) \}) & n>1.
    \end{cases}
\end{displaymath}
An object $G\in\mathcal{T}$ is called a \textbf{strong $\oplus$-generator} if $\mathcal{T}=\overline{\langle G \rangle}_n$ for some $n\geq 0$.

\begin{example}
    \label{ex:regular_ring_strong_generator}
    Let $X=\operatorname{Spec}(R)$ where $R$ is a regular ring of finite Krull dimension. Then $\overline{\langle\mathcal{O}_X \rangle}_{\dim X + 1} = D_{\operatorname{qc}}(X)$. On some level, this goes as far back as \cite{Kelly:1965}, \cite{Street:1973}, and \cite[Corollary 8.4]{Christensen:1998}; see \cite[Corollary 4.3.13]{Letz:2020} for a modern treatment.
\end{example}

\subsubsection{\texorpdfstring{$t$}{t}-structures}
\label{sec:prelim_t-structures}
 
We recall some background concerning $t$-structures, but the reader is referred to \cite{Beilinson/Berstein/Deligne/Gabber:2018}. When $\mathcal{T}$ admits small coproducts, let $\mathcal{T}^c$ denote the subcategory of compact objects in $\mathcal{T}$, i.e.\ consisting of those objects $c$ such that $\operatorname{Hom}(c,-)$ preserves coproducts. A pair of strictly full subcategories $\tau=(\mathcal{T}^{\leq 0},\mathcal{T}^{\geq 0})$ of $\mathcal{T}$ is called a \textbf{$t$-structure} if the following conditions are satisfied:
\begin{itemize}
    \item for all $A\in\mathcal{T}^{\leq 0}$ and $B\in\mathcal{T}^{\geq 0}[-1]$, one has $\operatorname{Hom}(A,B)=0$,
    \item $\mathcal{T}^{\leq 0}[1]\subseteq\mathcal{T}^{\leq 0}$ and $\mathcal{T}^{\geq 0}[-1]\subseteq\mathcal{T}^{\geq 0}$,
    \item for any $E\in\mathcal{T}$, there exists a distinguished triangle
    \begin{displaymath}
        A \to E \to B \to A[1]
    \end{displaymath}
    with $A\in\mathcal{T}^{\leq 0}$ and $B\in\mathcal{T}^{\geq 0}[-1]$.
\end{itemize}
For any integer $n$, define $\mathcal{T}^{\leq n}:= \mathcal{T}^{\leq 0}[-n]$ and $\mathcal{T}^{\geq n}:= \mathcal{T}^{\geq 0}[-n]$. 
A pair of $t$-structures $(\mathcal{T}^{\leq 0}_1, \mathcal{T}^{\geq 0}_1)$ and $(\mathcal{T}^{\leq 0}_2, \mathcal{T}^{\geq 0}_2)$ on $\mathcal{T}$ are called \textbf{equivalent} if there exists an $N\geq 0$ such that $\mathcal{T}^{\leq -N}_2 \subseteq \mathcal{T}^{\leq 0}_1 \subseteq \mathcal{T}^{\leq N}_2$. This forms an equivalence relation on the collection of $t$-structures on $\mathcal{T}$.
For example, the pairs $(\mathcal{T}^{\leq n}, \mathcal{T}^{\geq n})$ are all equivalent $t$-structure on $\mathcal{T}$.

An \textbf{aisle} on $\mathcal{T}$ is a strictly full subcategory $\mathcal{A}$ of $\mathcal{T}$ that is closed under positive shifts, extensions and such that the inclusion $\mathcal{A} \subseteq \mathcal{T}$ admits a right adjoint.
For any $t$-structure $(\mathcal{T}^{\leq 0},\mathcal{T}^{\geq 0})$, one has that $\mathcal{T}^{\leq 0}$ is an aisle. 
Conversely, any aisle $\mathcal{A}$ gives rise to a $t$-structure $(\mathcal{A},\mathcal{A}^{\perp}[1])$ (see \cite{Keller/Vossieck:1988}), where
$\mathcal{A}^\perp:= \{ T\in \mathcal{T}\mid \operatorname{Hom}(A,T)=0\text{ for all }A\in\mathcal{A}\}$.
For any $t$-structure, we therefore call $\mathcal{T}^{\leq 0}$ (resp.\ $\mathcal{T}^{\geq 0}$) the aisle (resp.\ coaisle) of the $t$-structure. It is useful to note that aisles are closed under direct summands \cite[Corollary 1.4]{AlonsoTarrio/JeremiasLopez/Salorio/Souto:2003}.

\begin{example}
    Assume that $\mathcal{T}$ admits small coproducts. Let $\mathcal{A}$ be a full subcategory of $\mathcal{T}^c$ closed under positive shifts. Denote by $\operatorname{Coprod}(\mathcal{A})$ the smallest strictly full subcategory of $\mathcal{T}$ that contains $\mathcal{A}$ which is closed under extensions and small coproducts. By \cite[Theorem 2.3.3 \& Remark 2.3.4]{Canonaco/Haesemeyer/Neeman/Stellari:2024} (which generalizes \cite[Theorem A.1 \& Proposition A.2]{AlonsoTarrio/LopezJeremias/Salorio:2003}), this construction defines an aisle in $\mathcal{T}$. We call the associated $t$-structure $\tau_{\mathcal{A}}$ the \textbf{$t$-structure compactly generated by $\mathcal{A}$}. If $\mathcal{A}=\{G[i]\mid i\ge 0\}$ for some compact object $G$; we denote the corresponding compactly generated $t$-structure by $\tau_G$. If $\mathcal{T}$ is compactly generated by a single object $G$, we define the \textbf{preferred equivalence class} to be the equivalence class of $t$-structures containing the $t$-structure compactly generated by $G$.
\end{example}

Let $F \colon \mathcal{T} \to \mathcal{T}^\prime$ be an exact functor between triangulated categories equipped with $t$-structures $(\mathcal{T}^{\leq 0}, \mathcal{T}^{\geq 0})$ and $((\mathcal{T}^\prime)^{\leq 0}, (\mathcal{T}^\prime)^{\geq 0})$. We say that $F$ is \textbf{right $t$-exact} if $
F(\mathcal{T}^{\leq 0}) \subseteq (\mathcal{T}^\prime)^{\leq 0}$,
and \textbf{left $t$-exact} if 
$F(\mathcal{T}^{\geq 0}) \subseteq (\mathcal{T}^\prime)^{\geq 0}$. If both conditions hold, then $F$ is \textbf{$t$-exact}.

\subsubsection{Recollements}
\label{sec:prelim_recollements}

We briefly recall the notion of a recollement. See \cite[\S 1.4]{Beilinson/Berstein/Deligne/Gabber:2018} for details. A \textbf{recollement} is a diagram of triangulated categories and exact functors of the form 
\begin{equation}
    \label{eq:recollement}
    \begin{tikzcd}
        {\mathcal{T}} && {\mathcal{K}} && {\mathcal{D}}
        \arrow["I"{description}, from=1-1, to=1-3]
        \arrow["{I_\lambda}"', bend right =25pt, from=1-3, to=1-1]
        \arrow["{I_\rho}", bend right =-25pt, from=1-3, to=1-1]
        \arrow["Q"{description}, from=1-3, to=1-5]
        \arrow["{Q_\lambda}"', bend right =25pt, from=1-5, to=1-3]
        \arrow["{Q_\rho}", bend right =-25pt, from=1-5, to=1-3]
    \end{tikzcd}
\end{equation}
satisfying:
\begin{itemize}
    \item $I_\lambda \dashv I \dashv I_\rho$ and $Q_\lambda \dashv Q \dashv Q_\rho$ (i.e.\ adjoint triples)
    \item $I, Q_\lambda, Q_\rho$ are fully faithful
    \item $\ker (Q)$ coincides with the essential image of $I$, i.e.\ the smallest strictly full subcategory containing $\{I(T) \mid T\in \mathcal{T}\}$.
\end{itemize}
As a consequence of the definition, there are distinguished triangles
\begin{displaymath}
    \begin{aligned}
        &(Q_\lambda \circ Q )(E)  \to E \to (I \circ I_\lambda )(E) \to (Q_\lambda \circ Q) (E)[1],
        \\& (I \circ I_\rho) (E) \to E \to (Q_\rho \circ Q) (E) \to (I \circ I_\rho )(E)[1]
    \end{aligned}
\end{displaymath}
which are functorial in $E\in \mathcal{K}$. In fact, the natural transformations between these functors are given by the (co)units of the relevant adjoint pairs. Since $Q_\lambda$, $Q$, $I$, and $I_\lambda$ are left adjoints, they preserve coproducts. Additionally, as $I$ and $Q$ admit right adjoints, and so preserve coproducts, $I_\lambda$ and $Q_\lambda$ preserve compact objects (see \cite[Theorem 5.1]{Neeman:1996}).

We may glue $t$-structures along recollements. Indeed, consider a recollement as above and suppose we are given aisles $\mathcal{A}_{\mathcal{T}}$ and $\mathcal{A}_{\mathcal{D}}$ on respectively  $\mathcal{T}$ and $\mathcal{D}$. Then the \textbf{glued $t$-structure} on $\mathcal{K}$ has an associated aisle given by 
\begin{displaymath}
    \mathcal{A}_{\mathcal{K}} := \{ E\in \mathcal{K} \mid I_\lambda (E)\in \mathcal{A}_{\mathcal{T}} \textrm{ and } Q(E)\in \mathcal{A}_{\mathcal{D}} \}.
\end{displaymath}
See \cite[\S 1.4.10]{Beilinson/Berstein/Deligne/Gabber:2018} for further details. 

\subsection{Algebraic stacks}
\label{sec:prelim_algebraic_stacks}

We collect some facts for algebraic stacks. Our conventions follow \cite{stacks-project}. An exception to this is the derived pullback/pushforward adjunction where we follow \cite[\S1]{Hall/Rydh:2017,Olsson:2007,Laszlo/Olsson:2008}. The symbols $X$, $Y$, etc.\ refer to schemes and algebraic spaces, whereas $\mathcal{X}$, $\mathcal{Y}$, etc.\ are for algebraic stacks.
We denote the underlying topological space of an algebraic stack $\mathcal{X}$ by $|\mathcal{X}|$.

\subsubsection{Derived categories and functors}
\label{sec:prelim_algebraic_stacks_categories_functors}

Let $\mathcal{X}$ be an algebraic stack. Denote by $\operatorname{Mod}(\mathcal{X})$ the Grothendieck Abelian category of sheaves of $\mathcal{O}_\mathcal{X}$-modules on the lisse-\'{e}tale topos of $\mathcal{X}$. The subcategories of quasi-coherent $\mathcal{O}_\mathcal{X}$-modules are denoted by $\operatorname{Qcoh}(\mathcal{X})$. Set $D(\mathcal{X}):=D(\operatorname{Mod}(\mathcal{X}))$ to be the derived category of $\operatorname{Mod}(\mathcal{X})$. Let $D_{\operatorname{qc}}(\mathcal{X})\subseteq D(\mathcal{X})$ be the full subcategory of complexes with quasi-coherent cohomology sheaves. If $\mathcal{X}$ is locally Noetherian, let $D_{\operatorname{coh}}(\mathcal{X})\subseteq D(\mathcal{X})$ consists of complexes with coherent cohomology. Also, denote $D^{\flat}_{\#}(\mathcal{X}):=D^{\flat}(\mathcal{X})\cap D_{\#}(\mathcal{X})$ for $\flat\in \{+,-,b,\geq n,\dots\}$ and $\#\in \{\operatorname{qc},\operatorname{coh}\}$. 

Let $f\colon \mathcal{X}\to \mathcal{Y}$ be a morphism of algebraic stacks. We refer to \cite[\S 1]{Hall/Rydh:2017} for the construction of the adjunction $\mathbf{L}f^\ast \colon D_{\operatorname{qc}}(\mathcal{Y})\rightleftarrows D_{\operatorname{qc}}(\mathcal{X}) \colon \mathbf{R}f_\ast$. We say $f$ is \textbf{concentrated} if it is quasi-compact, quasi-separated, and if the pushforward of any base change along a quasi-compact quasi-separated morphism $\mathcal{Z}\to \mathcal{X}$ has finite cohomological dimension. See \cite[\S 2]{Hall/Rydh:2017} for details.

One says a quasi-compact quasi-separated algebraic stack $\mathcal{X}$ is \textbf{concentrated} when the structure morphism $\mathcal{X} \to \operatorname{Spec}(\mathbb{Z})$ is such.
Given quasi-compact quasi-separated algebraic stacks with finitely presented inertia, being concentrated with quasi-finite diagonal is equivalent to the stack being tame; see e.g.\ \cite[Appendix A]{DeDeyn/Lank/ManaliRahul:2025}.

\subsubsection{Compacts and perfects}
\label{sec:prelim_algebraic_stacks_compacts}

On any ringed site, e.g.\ the lisse-\'{e}tale site of an algebraic stack, a complex is \textbf{strictly perfect} if it is a bounded complex with each term a direct summand of a finite free module. A complex is \textbf{perfect} if it is locally strictly perfect. Let $\operatorname{Perf}(\mathcal{X})$ denote the triangulated subcategory of $D_{\operatorname{qc}}(\mathcal{X})$ consisting of perfect complexes.

The following is a slight reformulation of \cite[Lemma 4.1]{Hall/Rydh:2017}, where we added the `fpqc' assumption that was missing in loc.\ cit. Following \cite[\href{https://stacks.math.columbia.edu/tag/022B}{Tag 022B}]{stacks-project}, we say a collection of morphism of algebraic stacks $\{f_i\colon\mathcal{X}_i\to\mathcal{X}\}_i$ is a \textbf{fpqc covering} if each $f_i$ is flat and for every quasi-compact open subset $\mathcal{U} \subseteq\mathcal{X}$ there exists a \emph{finite} number of quasi-compact opens $\mathcal{V}_j\subseteq \mathcal{X}_j$ with $\mathcal{U}=\cup_j f_j(\mathcal{V}_j)$. We say a morphisms $f\colon\mathcal{U}\to\mathcal{X}$ is fpqc if $\{f\colon\mathcal{Y}\to\mathcal{X}\}$ is a fpqc covering (in particular, $f$ is automatically surjective).

\begin{lemma}
    \label{lem:perfect_characterization}
    Let $\mathcal{X}$ be an algebraic stack. 
    For any $P\in D_{\operatorname{qc}}(\mathcal{X})$ the following are equivalent
    \begin{enumerate}
        \item\label{item:perf1} $P$ is perfect,
        \item\label{item:perf2} for every fpqc morphism $U \to \mathcal{X}$ from a scheme, $\mathbf{L}f^\ast P$ is perfect,
        \item\label{item:perf3} there exists a fpqc morphism $U \to \mathcal{X}$ from a scheme with $\mathbf{L}f^\ast P$ perfect,
        \item\label{item:perf4} there exists an fpqc covering $\{\operatorname{Spec}(R_i) \to \mathcal{X}\}_i$ such that each $\mathbf{R}\Gamma(\operatorname{Spec}(R_i),\mathbf{L}f^\ast P)$ is a strictly perfect complex of $R_i$-modules.
    \end{enumerate}
\end{lemma}

\begin{proof}
    That $\eqref{item:perf4}\implies\eqref{item:perf1}$ is \cite[Lemma 4.1]{Hall/Rydh:2017}. Indeed, the question is local on $\mathcal{X}$, so we may assume $\mathcal{X}$ is affine and use \cite[\href{https://stacks.math.columbia.edu/tag/068T}{Tag 068T}]{stacks-project}.  
    That $\eqref{item:perf1}\implies\eqref{item:perf2}\implies\eqref{item:perf3}\implies\eqref{item:perf4}$ is clear (pullback preserves perfect complexes, e.g.\ one can check this locally).
\end{proof}

In fact, the perfect complexes coincide with the dualizable (also known as, rigid objects) in $D_{\operatorname{qc}}(\mathcal{X})$ \cite[Lemma 4.3]{Hall/Rydh:2017}.
However, in general the perfect complexes and compact objects (in $D_{\operatorname{qc}}$) are not the same for quasi-compact quasi-separated algebraic stacks (as opposed to what happens for schemes). 
In fact, by \cite[Lemma 2.5(5) and Remark 4.6]{Hall/Rydh:2017} the following are equivalent: 
\begin{itemize}
    \item every perfect complex on $\mathcal{X}$ is compact, i.e.\ $\operatorname{Perf}(\mathcal{X})=D_{\operatorname{qc}}(\mathcal{X})^c$,
    \item $\mathcal{O}_{\mathcal{X}}$ is compact,
    \item $\mathcal{X}$ is concentrated. 
\end{itemize}
 
We note the following for future reference. Here, $\operatorname{\mathbf{R}\mathcal{H}\! \mathit{om}}_{\mathcal{O}_{\mathcal{X}}}$ denotes the internal Hom in $D_{\operatorname{qc}}(\mathcal{X})$. 

\begin{lemma}
    \label{lem:hall_pullback_sheaf_hom}
    Let $f \colon \mathcal{Y} \to \mathcal{X}$ be a concentrated morphism of algebraic stacks. Then the natural morphism
    \begin{displaymath}
        \mathbf{L}f^\ast \operatorname{\mathbf{R}\mathcal{H}\! \mathit{om}}_{\mathcal{O}_{\mathcal{X}}}(F,G)
        \to \operatorname{\mathbf{R}\mathcal{H}\! \mathit{om}}_{\mathcal{O}_{\mathcal{Y}}}(\mathbf{L}f^\ast F, \mathbf{L}f^\ast G)\quad\text{in $D_{\operatorname{qc}}(\mathcal{Y})$}
    \end{displaymath}
    is an isomorphism for any $F \in \operatorname{Perf}(\mathcal{X})$ and $G\in D_{\operatorname{qc}}(\mathcal{X})$.
\end{lemma}

\begin{proof}
    The problem is smooth local on the source and target (use \cite[Lemma 4.3]{Hall/Rydh:2017}). This allows us to reduce to the scheme setting.
    In this case, the claim follows by e.g.\ \cite[Proposition 22.70]{Gortz/Wedhorn:2023}.
\end{proof}

\subsubsection{Support}
\label{sec:prelim_algebraic_stacks_support}

For a scheme $X$ and $M\in\operatorname{Qcoh}(X)$, define 
\begin{displaymath}
    \operatorname{supp}(M):=\{x\in X \mid M_x \neq 0\}.
\end{displaymath}
More generally, for an algebraic stack $\mathcal{X}$ and $M\in\operatorname{Qcoh}(\mathcal{X})$ define 
\begin{displaymath}
    \operatorname{supp}(M):=p(\operatorname{supp}(p^\ast M))\subseteq|\mathcal{X}|
\end{displaymath}
where $p \colon U \to \mathcal{X}$ is any smooth surjective morphism from a scheme (one can check this is independent of  choice). 
Finally, for an object $E\in D_{\operatorname{qc}}(\mathcal{X})$, define the \textbf{(cohomological) support of $E$}, as
\begin{displaymath}
\operatorname{supp}(E):= \bigcup_{j\in\mathbb{Z}} \operatorname{supp}\left(\mathcal{H}^j (E)\right).    
\end{displaymath}

Let $Z\subseteq|\mathcal{X}|$ be a closed subset. We say $E\in D_{\operatorname{qc}}(\mathcal{X})$ is \textbf{supported on $Z$} if $\operatorname{supp}(E)\subseteq Z$. Define $D_{\operatorname{qc},Z}(\mathcal{X})$ as the full subcategory of $D_{\operatorname{qc}}(\mathcal{X})$ consisting of objects supported on $Z$. Similar categories are defined using the adornments $+$, $-$, $b$, etc. Furthermore, if $\mathcal{Z}\subseteq \mathcal{X}$ is a closed substack, then we say $E$ is \textbf{supported on $\mathcal{Z}$} when $E$ is supported on $|\mathcal{Z}|$. 

\subsubsection{Thomason condition}
\label{sec:prelim_algebraic_stacks_thomason}

Let $\beta$ be a cardinal.
We say a quasi-compact quasi-separated algebraic stack $\mathcal{X}$ satisfies the \textbf{$\beta$-Thomason condition} if $D_{\operatorname{qc}}(\mathcal{X})$ is compactly generated by a collection of cardinality $\leq \beta$ and for each quasi-compact open immersion $\mathcal{U}\hookrightarrow \mathcal{X}$, there exists a perfect complex $P$ over $\mathcal{X}$ with support $|\mathcal{X}|\setminus|\mathcal{U}|$. See \cite[Definition 8.1]{Hall/Rydh:2017}. We are primarily interested in the case $\beta = 1$. It follows by \cite[Lemma 4.10]{Hall/Rydh:2017} that for a $1$-Thomason algebraic stack, $D_{\operatorname{qc},Z}(\mathcal{X})$ is compactly generated by a single object whenever $Z$ has quasi-compact complement.

\section{Regularity for algebraic stacks}
\label{sec:regularity_for_stacks}

This section shows that regularity of algebraic stacks is a homological notion, akin to Serre's homological characterization of regular local rings. Let us recall the definition of regularity for algebraic stacks for convenience, following e.g.\ \cite[\href{https://stacks.math.columbia.edu/tag/04YE}{Tag 04YE}]{stacks-project}. For defining regularity at a point, we make use of the fact that if $(X,x)$ is a germ of a scheme the property `$\mathcal{O}_{X,x}$ is regular' is a smooth local property of germs by \cite[Proposition 17.5.8]{EGAIV4:1967}.

\begin{definition}
    \label{def:regular_algebraic_stack}
    Let $\mathcal{X}$ be an algebraic stack. A point $x\in |\mathcal{X}|$ is called \textbf{regular} if there exists a smooth morphism $f \colon U \to \mathcal{X}$ from a scheme and a point $u\in U$ with $f(u) = x$ such that $\mathcal{O}_{U,u}$ is a regular local ring. More generally, $\mathcal{X}$ is called \textbf{regular} if there exists a smooth surjective morphism $U \to \mathcal{X}$ from a regular scheme.
\end{definition}

We could have given a ‘point-wise’ definition for the regularity of an algebraic stack, as the following lemma shows. Moreover, for decent algebraic stacks it suffices to look at closed points. Recall that an algebraic stack is called \textbf{decent} when it has enough `quasi-compact points' (see \cite[\href{https://stacks.math.columbia.edu/tag/0GW0}{Tag 0GW0}]{stacks-project}). For us, it suffices to know that any algebraic stack with quasi-compact diagonal (in particular, any quasi-separated algebraic stack) is decent \cite[\href{https://stacks.math.columbia.edu/tag/v}{Tag 0GW2}]{stacks-project}. This includes Noetherian algebraic stacks (because they are quasi-separated by definition). 

\begin{lemma}
    \label{lem:regular_algebraic_stack_equivalence}
    An algebraic stack $\mathcal{X}$ is regular if, and only if, every point $x\in |\mathcal{X}|$ is regular.
    Moreover, when $\mathcal{X}$ is quasi-compact and decent, 
    it suffices to only check the closed points.
    In addition, any regular algebraic stack is locally Noetherian and normal.
\end{lemma}

\begin{proof}
    Choose a smooth surjective morphism from a scheme $U\to \mathcal{X}$.
    Then, 
    \begin{displaymath}
        \begin{aligned}
            \mathcal{X}\text{ is regular} &\iff U \text{ is regular}\qquad &&\text{(\cite[\href{https://stacks.math.columbia.edu/tag/04YF}{Tag 04YF}]{stacks-project})} \\
            &\iff \text{all }u\in U\text{ are regular}\qquad &&\text{(\cite[\href{https://stacks.math.columbia.edu/tag/02IT}{Tag 02IT}]{stacks-project})} \\
            &\iff \text{all }x\in |\mathcal{X}|\text{ are regular}\qquad &&\text{(\cite[\href{https://stacks.math.columbia.edu/tag/04YI}{Tag 04YI}]{stacks-project})}.
        \end{aligned}
    \end{displaymath}
    To see that for quasi-compact decent algebraic stacks, it suffices to look at closed points. Observe, by \cite[\href{https://stacks.math.columbia.edu/tag/0GVZ}{Tag 0GVZ}]{stacks-project}, the morphism $U\to \mathcal{X}$ lifts generalizations and that there are enough closed points on quasi-compact decent algebraic stacks (i.e.\ any point specializes to a closed point).
    Indeed, for $y$ an arbitrary point of $|\mathcal{X}|$, there exists\footnote{
    It is a topological fact that any nonempty, quasi-compact, Kolmogorov topological space contains a closed point \cite[\href{https://stacks.math.columbia.edu/tag/005E}{Tag 005E}]{stacks-project}. Any decent algebraic stack has Kolmogorov underlying topological space \cite[\href{https://stacks.math.columbia.edu/tag/0GW7}{Tag 0GW7}]{stacks-project}. 
    } a closed point $x$ with $x\in \overline{\{y\}}$ (i.e.\ $y$ is a generalization of $x$).
    As $x$ is closed, and thus a regular point, there exists a $u$ lying over $x$ with $\mathcal{O}_{U,u}$ regular. 
    However, as generalizations lift by loc.\ cit.,  there exists a $v$ over $y$ such that $v$ is a generalization of $u$. Hence, it follows that $\mathcal{O}_{U,v}$ is regular, i.e.\ $y$ is regular.

    The last claim follows from the fact that regular schemes are locally Noetherian and normal, and that these properties are local for the smooth topology.
\end{proof}

\begin{remark}
    The `quasi-compact and decent' condition above is needed in order to have `enough closed points', for which we use a purely topological fact.
    For schemes one can do a bit `better': any locally Noetherian scheme has enough closed points \cite[\href{https://stacks.math.columbia.edu/tag/02IL}{Tag 02IL}]{stacks-project}.
    It would be interesting to know whether the same holds for algebraic stacks. 
    That is, does any locally Noetherian algebraic stack have enough closed points?
\end{remark}

This leads to the usual definitions for the regular and singular locus. 

\begin{definition}
\label{def:regular_locus}
    Let $\mathcal{X}$ be a locally Noetherian algebraic stack.
    \begin{enumerate}
        \item The \textbf{regular locus of $\mathcal{X}$} is defined as $\operatorname{reg} (\mathcal{X}):=\{ x\in |\mathcal{X}| \mid \text{$x$ is regular}\}$.
        \item The \textbf{singular locus of $\mathcal{X}$} is defined as $\operatorname{sing} (\mathcal{X}):=|\mathcal{X}|\setminus \operatorname{reg} (\mathcal{X})$.
    \end{enumerate}
\end{definition}

The following shows that one can also describe the regular/singular locus using smooth covers and that it is closed under generalizations when the algebraic stack is decent.

\begin{lemma}
    \label{lem:regular_locus_stable_generalization}
    Let $\mathcal{X}$ be a locally Noetherian algebraic stack. 
    Suppose $f \colon U \to \mathcal{X}$ be a smooth surjective morphism from a scheme $U$.
    Then $\operatorname{reg} (\mathcal{X})=f(\operatorname{reg}(U))$.
    Moreover, when $\mathcal{X}$ is decent the regular locus is closed under generalization.
\end{lemma}

\begin{proof}
    The first statement follows immediately from the definition of a regular point. 
    The second statement follows by reasoning along the lines of the proof of  \Cref{lem:regular_algebraic_stack_equivalence}, i.e.\ that one can lift generalizations along smooth covers when the algebraic stack is decent \cite[\href{https://stacks.math.columbia.edu/tag/0GVZ}{Tag 0GVZ}]{stacks-project}, combined with that the regular locus of a locally Noetherian scheme is closed under generalization.
\end{proof}

We now come to the main proposition of this section.

\begin{proposition}
    \label{prop:closed_subset_regular_locus}
    Let $\mathcal{X}$ be a decent quasi-compact
    locally Noetherian algebraic stack. For any closed subset $Z\subseteq |\mathcal{X}|$, the following are equivalent:
    \begin{enumerate}
        \item\label{item:closedreg1} every point in $Z$ is regular, i.e.\ $Z\subseteq \operatorname{reg}(\mathcal{X})$,
        \item\label{item:closedreg2} every closed point in $Z$ is regular, 
        \item\label{item:closedreg3} $D^b_{\operatorname{coh},Z}(\mathcal{X})=\operatorname{Perf}_Z (\mathcal{X})$.
    \end{enumerate}
\end{proposition}

Before giving the proof of this proposition we include a proof of the analogous result for schemes due to the lack of finding a reference.

\begin{lemma}
    \label{lem:scheme_perf_Z=DbCohZ_iff_Zinreg}
    Let $X$ be a locally Noetherian scheme. Suppose $Z \subseteq X$ a closed subset. 
    Then $Z \subseteq \operatorname{reg}(X)$ if, and only if, $\operatorname{Perf}_{Z}(X) = D^b_{\operatorname{coh},Z}(X)$.
\end{lemma}

\begin{proof}
    First, assume $Z \subseteq \operatorname{reg}(X)$. Choose $E\in D^b_{\operatorname{coh},Z}(X)$. 
    Then, for any $x \in X$, either $E_{x}$ is perfect or zero (hence also perfect). This implies $E$ is perfect, since one can assume $X$ is affine as this is a local question and invoke \cite[Theorem 4.1]{Avramov/Iyengar/Lipman:2010}.
    Thus, $\operatorname{Perf}_{Z}(X) = D^b_{\operatorname{coh},Z}(X)$.

    For the converse direction, choose a closed point $x\in Z$.
    As the point is closed, the skyscraper sheaf $\kappa(x)$ (the pushforward of the residue field $\kappa(x):=\mathcal{O}_{X,x}/\mathfrak{m}_x$ along the closed immersion $\operatorname{Spec}(\kappa(x)) \to X$) is coherent. Hence, $\kappa(x)\in \operatorname{Perf}_Z(X)$ by our hypothesis. 
    By taking stalks, we see that $\kappa(x)$ has finite projective dimension as an $\mathcal{O}_{X,x}$-module. So, it follows that $x\in \operatorname{reg}(X)$ (see e.g.\ \cite[\href{https://stacks.math.columbia.edu/tag/00OC}{Tag 00OC}]{stacks-project}). 
    As locally Noetherian schemes have enough closed points by \cite[\href{https://stacks.math.columbia.edu/tag/02IL}{Tag 02IL}]{stacks-project} and the localization of a regular local ring remains regular, it follows that $Z \subseteq \operatorname{reg}(X)$.
\end{proof}

\begin{proof}[Proof of \Cref{prop:closed_subset_regular_locus}]
    We first show $\eqref{item:closedreg1}\implies\eqref{item:closedreg3}$. For this consider a smooth surjective morphism $f\colon U \to \mathcal{X}$ from a Noetherian scheme. By $\eqref{item:closedreg1}$ and \Cref{lem:regular_locus_stable_generalization}, it follows $f^{-1}(Z)\subseteq\operatorname{reg}(U)$. Consequently, by \Cref{lem:scheme_perf_Z=DbCohZ_iff_Zinreg}, it follows that $D^b_{\operatorname{coh},f^{-1}(Z)}(U)=\operatorname{Perf}_{f^{-1}(Z)} (U)$. Applying \Cref{lem:perfect_characterization} it follows that every object in $D^b_{\operatorname{coh},Z}(\mathcal{X})$ is perfect, which shows $\eqref{item:closedreg3}$ as desired.

    To see that $\eqref{item:closedreg3}\implies\eqref{item:closedreg2}$ fix a closed point $z\in Z$ and choose a smooth surjective morphism $U \to \mathcal{X}$ from an \emph{affine} Noetherian scheme. 
    Let $i\colon \mathcal{Z}_z \hookrightarrow \mathcal{X}$ be the residual gerbe at $z$ (which exists by \cite[\href{https://stacks.math.columbia.edu/tag/0H22}{Tag 0H22}]{stacks-project}). Consider the fibered square
    \begin{displaymath}
        \begin{tikzcd}[ampersand replacement=\&]
            \llap{$V:=$\;}{U\times_\mathcal{X} \mathcal{Z}_z} \& {\mathcal{Z}_z} \\
            U \& {\mathcal{X}}\rlap{ .}
            \arrow["{g}", from=1-1, to=1-2]
            \arrow["{j}"', hook, from=1-1, to=2-1]
            \arrow["i", from=1-2, hook, to=2-2]
            \arrow["f"', from=2-1, to=2-2]
        \end{tikzcd}
    \end{displaymath}
    Note that $i$ is a closed immersion from a regular algebraic stack \cite[\href{https://stacks.math.columbia.edu/tag/0H27}{Tag 0H27} \& \href{https://stacks.math.columbia.edu/tag/06MV}{Tag 06MV}]{stacks-project}. 
    Moreover, by base change, $g$ is a smooth surjective morphism and $j$ is a closed immersion.
    It follows that $Z$ is a (nonempty, as $|\mathcal{Z}_z|$ is a point) regular closed subscheme of $U$.

    Choose any closed point $v \in V$ and let $u:=j(v)$.
    As $f(u)=z$, to show $z$ is regular, it suffices, by definition, to show that $u\in U$ is regular.
    That is, we must show the residue field $\kappa(u):=\mathcal{O}_{U,u}/\mathfrak{m}_u=\mathcal{O}_{Z,z}/\mathfrak{m}_z=\kappa(z)$ is a perfect $\mathcal{O}_{U,u}$-module.
    As $v\in V$ is regular, we know that $\kappa(v)$ is a perfect $\mathcal{O}_{V,v}$-module. Thus it suffices to show that $\mathcal{O}_{V,v}$ is perfect as $\mathcal{O}_{U,u}$-module. 
    Now, $i_\ast \mathcal{O}_{\mathcal{Z}_z}$ is supported in $Z$, it is simply supported at $z\in Z$, and so is perfect by our assumption $\eqref{item:closedreg3}$.
    Consequently, using flat base change (see \cite[Corollary 4.13]{Hall/Rydh:2017}), it follows that 
    \begin{displaymath}
         \left(j_\ast \mathcal{O}_V \right)_u = \left(\mathbf{R} j_\ast \mathcal{O}_V \right)_u = \left(\mathbf{R}j_\ast\mathbf{L}g^\ast \mathcal{O}_{\mathcal{Z}_z} \right)_u = \left(\mathbf{L}f^\ast \mathbf{R}i_\ast \mathcal{O}_{\mathcal{Z}_z}\right)_u = \left(\mathbf{L}f^\ast i_\ast \mathcal{O}_{\mathcal{Z}_z}\right)_u
    \end{displaymath}
    is a perfect $\mathcal{O}_{U,u}$-module.
    However, this module is exactly $\mathcal{O}_{V,v}$ and so $z$ is regular as was to be shown.
    
    Lastly, $\eqref{item:closedreg2}\implies\eqref{item:closedreg1}$ follows from the fact the regular locus is closed under generalizations by \Cref{lem:regular_locus_stable_generalization} and that $\mathcal{X}$ has enough closed points (c.f.\ the proof of \Cref{lem:regular_algebraic_stack_equivalence}).
\end{proof}

\begin{corollary}
    \label{cor:regularity_in_terms_of_perfectness}
    Let $\mathcal{X}$ be a decent quasi-compact locally Noetherian algebraic stack. Then $\mathcal{X}$ is regular if, and only if, $\operatorname{Perf}(\mathcal{X})=D^b_{\operatorname{coh}}(\mathcal{X})$.
\end{corollary}

\begin{proof}
    This is a special case of \Cref{prop:closed_subset_regular_locus} with $Z=|\mathcal{X}|$.
\end{proof}

\begin{example}
    \label{ex:C2}
    Let $G$ be a finite group acting on a regular Noetherian ring $R$. Consider the quotient stack $\mathcal{X}:=[\operatorname{Spec}(R)/G]$.
    Recall that
    \begin{displaymath}
        D_{\operatorname{qc}}(\mathcal{X})\cong D(\operatorname{Mod}(RG))    
    \end{displaymath}
    where the latter is the category of left modules over the skew group ring $RG$. 
    As $R$ is regular, so is $\mathcal{X}$ (see \Cref{{def:regular_algebraic_stack}}) and thus $\operatorname{Perf}(\mathcal{X})= D^b_{\operatorname{coh}}(\mathcal{X})$.
    However, under the above equivalence $D_{\operatorname{qc}}(\mathcal{X})^c$ corresponds to the category of perfect $RG$-modules  $\operatorname{Perf}(RG)$ which generally will not equal $D^b(\operatorname{mod}(RG))$ (which corresponds to $D^b_{\operatorname{coh}}(\mathcal{X})$); e.g.\ take $R$ to be a field of positive characteristic dividing the order of $G$.
    (The perfect complexes over $\mathcal{X}$ correspond to complexes of $RG$-modules that are perfect as $R$-modules.)
\end{example}

\section{Properties of the regular locus}
\label{sec:openness_locus}

In this section, we extend to algebraic stacks results relating the existence of classical generators (of the bounded derived category) to the the openness/interior properties of the regular locus.
This was previously proven for schemes in \cite{Iyengar/Takahashi:2019,Dey/Lank:2024a}. We start with a few lemmas.

\begin{lemma}
    \label{lem:support_union_closed_ses}
    Let $\mathcal{X}$ be a Noetherian algebraic stack. Consider closed subsets  $Z$ and $Z^\prime$ of $|\mathcal{X}|$. 
    Suppose $E$ is a coherent sheaf over $\mathcal{X}$ with
    $\operatorname{supp}(E)= Z \cup Z^\prime$. 
    Then there exists a short exact sequence of coherent sheaves
    \begin{displaymath}
        0 \to A \to E \to B \to 0
    \end{displaymath}
    with $\operatorname{supp}(A)\subseteq Z^\prime$ and $ \operatorname{supp}(B)\subseteq Z$.
\end{lemma}

\begin{proof}
    In the case $\mathcal{X}$ is a scheme see e.g.\ \cite[\href{https://stacks.math.columbia.edu/tag/01YD}{Tag 01YD}]{stacks-project}. As for the general case, note that all constructions in loc.\ cit.\ can be done on an algebraic stack. Hence, the necessary checks can be reduced to the scheme case by looking along a smooth cover.
\end{proof}

\begin{lemma}
    \label{lem:closed_immersion_generates_support}
    Let $i\colon \mathcal{Z} \hookrightarrow \mathcal{X}$ be a closed immersion to a Noetherian algebraic stack $\mathcal{X}$. Then $D^b_{\operatorname{coh},|\mathcal{Z}|} ( \mathcal{X} )=\langle \mathbf{R} i_\ast D^b_{\operatorname{coh}}(\mathcal{Z}) \rangle$.
\end{lemma}
  
\begin{proof}
    The proof is identical to the scheme case (see e.g.\ \cite[Proposition 4.6]{Elagin/Lunts/Schnurer:2020} or \cite[Lemma 3.4]{Hall:2023}) using a suitable stacky versions of \cite[\href{https://stacks.math.columbia.edu/tag/01Y9}{Tag 01Y9} \& \href{https://stacks.math.columbia.edu/tag/087T}{Tag 087T}]{stacks-project}.
\end{proof}

\begin{lemma}\label{lem:classical_generation_from_irreducible_components}
    Let $\mathcal{X}$ be a Noetherian algebraic stack.  Suppose $|\mathcal{X}| = \cup_{r=1}^n |\mathcal{Z}_r|$ where $i_r\colon \mathcal{Z}_r \hookrightarrow \mathcal{X}$ are closed immersions such that each $D^b_{\operatorname{coh}}(\mathcal{Z}_r)$ admits a classical generator $G_r$. 
    Then $D^b_{\operatorname{coh}}(\mathcal{X})$ is classically generated by $G:=\oplus_{r=1}^n \mathbf{R} (i_r)_\ast G_r$.
\end{lemma}

\begin{proof}
    By induction on $n$, we may assume $n=2$. Note the case $n=1$ follows from \Cref{lem:closed_immersion_generates_support}.
    Choose $E\in D^b_{\operatorname{coh}}(\mathcal{X})$. As any object in $D^b_{\operatorname{coh}}(\mathcal{X})$ is finitely built by its cohomology sheaves, we may assume that $E$ is a coherent sheaf concentrated in degree zero. Next, by \Cref{lem:support_union_closed_ses}, we may assume that $E$ is either an object of $D^b_{\operatorname{coh},|\mathcal{Z}_1|}(\mathcal{X})$ or $D^b_{\operatorname{coh},|\mathcal{Z}_2|}(\mathcal{X})$ in which case the claim follows from \Cref{lem:closed_immersion_generates_support}.
\end{proof}
  
The following is the natural extension of the scheme theoretic notions, see e.g.\ \cite[\href{https://stacks.math.columbia.edu/tag/07P6}{Tag 07P6} \& \href{https://stacks.math.columbia.edu/tag/07R2}{Tag 07R2}]{stacks-project}.

\begin{definition}
    Let $\mathcal{X}$ be a locally Noetherian algebraic stack. 
    Then we say $\mathcal{X}$ is
    \begin{enumerate}
        \item \textbf{$J\textrm{-}0$} if the regular locus $\operatorname{reg}(\mathcal{X})$ contains a nonempty open,
        \item \textbf{$J\textrm{-}1$} if the regular locus $\operatorname{reg}(\mathcal{X})$ is open,
        \item \textbf{$J\textrm{-}2$} if for every morphism $\mathcal{Y}\to \mathcal{X}$ which is locally of finite type, the regular locus $\operatorname{reg}(\mathcal{Y})$ is open.
    \end{enumerate}
\end{definition}

\begin{lemma}
    \label{lem:existence_classical_generator_implies_open_cohomological_regular_locus}
    Let $\mathcal{X}$ be a Noetherian algebraic stack. If $D^b_{\operatorname{coh}}(\mathcal{X})$ admits a classical generator, then the regular locus $\operatorname{reg}(\mathcal{X})$ is open.
\end{lemma}

\begin{proof}
    Suppose $D^b_{\operatorname{coh}}(\mathcal{X})=\langle G\rangle$ and let $f\colon U \to \mathcal{X}$ be a smooth surjective morphism from an affine scheme (which exists by \cite[\href{https://stacks.math.columbia.edu/tag/04YC}{Tag 04YC}]{stacks-project}). 
    Set 
    \begin{displaymath}
        S :=\{ u\in U \mid (\mathbf{L}f^\ast G)_u\text{ is a perfect $\mathcal{O}_{U,u}$-module} \}
    \end{displaymath}
    and note that it is generalization closed as the localisation of a perfect complex remains perfect.
    It follows from \cite[Proposition 3.5]{Letz:2021} that $S$ is an open subset, and so $f(S)$ is an open substack of $\mathcal{X}$.
    Therefore, it suffices to show that $f(S)= \operatorname{reg}(\mathcal{X})$.    
    For this observe that `$\supseteq$' follows from the definition of a regular point.
    To see `$\subseteq$', note that $G|_{f(S)}$ is a classical generator for $D^b_{\operatorname{coh}}(f(S))$ (restricting to an open substack is essentially surjective on coherent sheaves) and that $G|_{f(S)}$ is perfect by \Cref{lem:perfect_characterization}.
    Consequently, $D^b_{\operatorname{coh}}(f(S))=\operatorname{Perf}(f(S)$ and $f(S)$ is regular by \Cref{prop:closed_subset_regular_locus}, showing the required inclusion.
\end{proof}

For the next proposition, recall that a Noetherian algebraic stack $\mathcal{X}$ being integral is equivalent to $\mathcal{X}$ being reduced with $|\mathcal{X}|$ irreducible \cite[\href{https://stacks.math.columbia.edu/tag/0GWD}{Tag 0GWD}]{stacks-project}. 

\begin{proposition}
    \label{prop:j_condition_stacks}
    Let $\mathcal{X}$ be a concentrated Noetherian algebraic stack satisfying the $1$-Thomason condition. 
    Then the following are equivalent:
    \begin{enumerate}
        \item\label{item:jcond1} every integral closed substack $\mathcal{Z}$ of $\mathcal{X}$ is $J\textrm{-}0$,
        \item\label{item:jcond2} $D^b_{\operatorname{coh}}(\mathcal{Z})$ admits a classical generator for every integral closed substack $\mathcal{Z}$ of $\mathcal{X}$.
    \end{enumerate}
    Moreover, if any of these conditions hold, then $D^b_{\operatorname{coh}}(\mathcal{Y})$ admits a classical generator for any closed substack $\mathcal{Y}$ of $\mathcal{X}$.
\end{proposition}

\begin{proof}
    It is clear from  \Cref{lem:existence_classical_generator_implies_open_cohomological_regular_locus} that $\eqref{item:jcond2}\implies \eqref{item:jcond1}$. So, we check the converse. In fact, we show if $\eqref{item:jcond1}$ holds, then $D^b_{\operatorname{coh}}(\mathcal{Y})$ admits a classical generator for any closed substack $\mathcal{Y}$ of $\mathcal{X}$. This will be done by Noetherian induction. There is nothing to show when $\mathcal{Y}$ is empty, and so, we assume it is not.
    As $\mathcal{Y}$ is a Noetherian stack, $|\mathcal{Y}|$ has a finite number of irreducible closed components. Let us endow each of these with their reduced induced algebraic stack structure (\cite[\href{https://stacks.math.columbia.edu/tag/050C}{Tag 050C}]{stacks-project}). Then, by applying \Cref{lem:classical_generation_from_irreducible_components}, we may assume $\mathcal{Y}$ is integral. In such a case, the hypothesis ensures $\mathcal{Y}$ is $J\textrm{-}0$. 
    
    Observe, if $\mathcal{Y}$ is regular, that $D^b_{\operatorname{coh}}(\mathcal{Y})=\operatorname{Perf}(\mathcal{Y})$ admits a classical generator. Indeed, closed substacks of $1$-Thomason stacks are again $1$-Thomason (see e.g.\ \cite[Observation 5.6]{Neeman:2023b} or \cite[Lemma 4.4]{Lank:2026}). So, we may assume $\mathcal{Y}$ is not regular. Then, by the $J\textrm{-}0$ hypothesis, there exists a nonempty open immersion $j\colon \mathcal{U} \to \mathcal{Y}$ with $|\mathcal{U}|\subseteq \operatorname{reg}(\mathcal{Y})\neq |\mathcal{Y}|$. Here, $\mathcal{U}$ is regular, and so $D^b_{\operatorname{coh}}(\mathcal{U})$ admits a classical generator by the previous reasoning. 
    Let $\mathcal{Z} \subseteq \mathcal{X}$ be the closed algebraic stack obtained by endowing the closed subset $|\mathcal{X}|\setminus |\mathcal{U}|$ with its reduced induced structure.
    By the induction hypothesis, 
    we know that $D^b_{\operatorname{coh}}(\mathcal{Z})$ admits a classical generator. Hence, \Cref{lem:classical_generation_from_irreducible_components} ensures that $D^b_{\operatorname{coh},|\mathcal{Z}|}(\mathcal{X})$ admits a classical generator. Now, from \cite[Lemma B.1]{Hall/Lamarche/Lank/Peng:2025}, there is a Verdier localization sequence
    \begin{displaymath}
        D^b_{\operatorname{coh},|\mathcal{Z}|}(\mathcal{X}) \longrightarrow D^b_{\operatorname{coh}}(\mathcal{X}) \xrightarrow{\mathbf{L}j^\ast} D^b_{\operatorname{coh}}(\mathcal{U}).
    \end{displaymath}
    Consequently, it follows that $D^b_{\operatorname{coh}}(\mathcal{X})$ admits a classical generator (see e.g.\ \cite[Proposition 4.3]{Bergh/Lunts/Schnurer:2016}).   
\end{proof}

\begin{remark}
    In addition, the conditions of the theorem are equivalent to same statement of $\eqref{item:jcond2}$, but with $D_{\operatorname{sg}}(\mathcal{Z}):=D^b_{\operatorname{coh}}(\mathcal{Z})/\operatorname{Perf}(\mathcal{Z})$, the singularity category of $\mathcal{Z}$, instead of $D^b_{\operatorname{coh}}(\mathcal{Z})$. This follows from abstract nonsense using that the singularity category is a Verdier quotient (and using the fact that by the $1$-Thomason conditions $\operatorname{Perf}(\mathcal{Z})$ admits a classical generator).
\end{remark}

As a corollary we obtain a generalization of \cite[Theorem 4.15]{Elagin/Lunts/Schnurer:2020} in the setting of algebraic stacks.

\begin{corollary}
    \label{cor:ELS_j_conditions}
    Let $\mathcal{X}$ be a Noetherian $J\textrm{-}2$ concentrated algebraic stack with quasi-finite and separated diagonal. 
    Then $D^b_{\operatorname{coh}}(\mathcal{X})$ admits a classical generator.
\end{corollary}

\begin{proof}
    As $\mathcal{X}$ is $J\textrm{-}2$, every closed substack has open regular locus, and so the claim follows from \Cref{prop:j_condition_stacks}.
\end{proof}

\section{Characterization via strong generators}
\label{sec:detect_via_strong_generator}

This section proves an analogue of \cite[Theorem 3.1.4]{Bondal/VandenBergh:2003} and \cite[Theorem 0.5]{Neeman:2021a} for `nice' Noetherian algebraic stacks. We start with some lemmas and recall some terminology along the way.
Recall a Noetherian algebraic stack has \textbf{finite Krull dimension} if there exists a smooth surjective morphisms from a quasi-compact scheme of finite Krull dimension. It can be verified that this is independent of chosen cover using \cite[\href{https://stacks.math.columbia.edu/tag/0AFF}{Tag 0AFF} and \href{https://stacks.math.columbia.edu/tag/0AFI}{Tag 0AFI}]{stacks-project}. 
However, for the convenience of reader, we spell out the details and relate this with the notion of Krull dimension for a locally Noetherian algebraic stack as defined in \cite[\href{https://stacks.math.columbia.edu/tag/0AFP}{Tag 0AFP}]{stacks-project}.

\begin{lemma}
    \label{lem:stack_krull_dimension_finiteness_Noetherian}
    Let $\mathcal{X}$ be a quasi-compact locally Noetherian algebraic stack. Then the following are equivalent:
    \begin{enumerate}
        \item \label{lem:stack_krull_dimension_finiteness_Noetherian2} for every smooth surjective morphisms $U \to \mathcal{X}$ from a quasi-compact scheme, the scheme has finite Krull dimension,
        \item \label{lem:stack_krull_dimension_finiteness_Noetherian3} there exists a smooth surjective morphisms $U \to \mathcal{X}$ from a quasi-compact scheme of finite Krull dimension,
        \item \label{lem:stack_krull_dimension_finiteness_Noetherian1} $\mathcal{X}$ has finite Krull dimension, i.e.\ $\dim (\mathcal{X})$ as defined in \cite[\href{https://stacks.math.columbia.edu/tag/0AFP}{Tag 0AFP}]{stacks-project} is finite.
    \end{enumerate}
\end{lemma}

\begin{proof}
    Clearly $\eqref{lem:stack_krull_dimension_finiteness_Noetherian2}\implies \eqref{lem:stack_krull_dimension_finiteness_Noetherian3}$, whereas \cite[\href{https://stacks.math.columbia.edu/tag/0DRP}{Tag 0DRP}]{stacks-project} tells us $\eqref{lem:stack_krull_dimension_finiteness_Noetherian3} \implies \eqref{lem:stack_krull_dimension_finiteness_Noetherian1}$. 
    So, it rests to show $\eqref{lem:stack_krull_dimension_finiteness_Noetherian1}\implies \eqref{lem:stack_krull_dimension_finiteness_Noetherian2}$. 
    Consider a smooth surjective morphism $f\colon U \to \mathcal{X}$ from a quasi-compact scheme. 
    By \cite[\href{https://stacks.math.columbia.edu/tag/0DRQ}{Tag 0DRQ}]{stacks-project}, the function $F\colon|U|\to \mathbb{Z}$ given by $t\mapsto \dim_u (U_{f(u)})$ is locally constant. The `relative dimension' $\dim_u (U_{f(u)})$ is defined in \cite[\href{https://stacks.math.columbia.edu/tag/0DRG}{Tag 0DRG}]{stacks-project}. (That the function attains values in $\mathbb{Z}$ follows, e.g., from combining \cite[\href{https://stacks.math.columbia.edu/tag/02G1}{Tag 02G1}, \href{https://stacks.math.columbia.edu/tag/0DRG}{Tag 0DRG} (3) and \href{https://stacks.math.columbia.edu/tag/0DRN}{Tag 0DRN}]{stacks-project}.)
    As $|U|$ is quasi-compact, we know that $F$ attains only a finite number of values, and so, there exists an integer $N$ with $ F (u) \leq N$ for all $u \in |U |$.
    Consequently, by \cite[\href{https://stacks.math.columbia.edu/tag/0DRP}{Tag 0DRP}]{stacks-project}, we see that $U$ has finite Krull dimension.
\end{proof}

\begin{lemma}
    \label{lem:descent_strong_oplus_bounded_coherent_stacks}
    Let $f\colon \mathcal{Y} \to \mathcal{X}$ be a morphism of finite type between concentrated separated Noetherian algebraic stacks. 
    If $D_{\operatorname{qc}}(\mathcal{Y})$ admits a strong $\oplus$-generator with bounded and coherent cohomology, then so does $D_{\operatorname{qc}}(\mathcal{X})$.
\end{lemma}

\begin{proof}
    First, note $f$ is concentrated by \cite[Lemma 1.5(4)]{Hall/Rydh:2017}.
    From \cite[Theorem 6.3]{Hall/Lamarche/Lank/Peng:2025}, there exists an integer $n\geq 0$ with $ D_{\operatorname{qc}}(\mathcal{X}) = \overline{\langle \mathbf{R}f_{\ast} D_{\operatorname{qc}}(\mathcal{Y}) \rangle}_n$.
    Moreover, our assumption ensures there exists an $H\in D_{\operatorname{coh}}^b (\mathcal{Y})$ and $m\geq 0$ with $D_{\operatorname{qc}}(\mathcal{Y}) = \overline{\langle H \rangle}_m$. 
    Lastly, it follows from \cite[Lemma 6.8]{DeDeyn/Lank/ManaliRahul:2025} that there exists a $G\in D^b_{\operatorname{coh}}(\mathcal{X})$ and $l\geq 0$ with $\mathbf{R}f_{\ast} H \in \overline{\langle G \rangle}_l$.  
    Putting this all together, we see that $D_{\operatorname{qc}}(\mathcal{X}) = \overline{\langle G \rangle}_{mnl}$, which completes the proof.
\end{proof}

The main result of this section requires a technical condition on the algebraic stack known as `satisfies approximation by compacts'.
In \cite{Hall/Lamarche/Lank/Peng:2025} a stacky analogue of Lipman--Neeman's \cite{Lipman/Neeman:2007} approximation by perfect complexes was introduced (see also  or \cite[\href{https://stacks.math.columbia.edu/tag/08EL}{Tag 08EL} \& \href{https://stacks.math.columbia.edu/tag/08HH}{Tag 08HH}]{stacks-project} for the case of schemes and algebraic spaces).
A subtlety is that compacts and perfect objects do not coincide in general for algebraic stacks; however this does not matter for concentrated algebraic stacks. 
It is shown in \cite[Corollary 5.4]{Hall/Lamarche/Lank/Peng:2025} that algebraic stacks with quasi-finite and separated diagonal satisfy approximation by compacts. We refer the reader to loc.\ cit.\ for details.

\begin{lemma}
    \label{lem:perf_strong_implies_dbcoh=perf}
    Let $\mathcal{X}$ be a concentrated Noetherian algebraic stack satisfying approximation by compacts.
    For any $G\in D^b_{\operatorname{coh}}(\mathcal{X})$ and $n\geq 0$, the following are equivalent:
    \begin{enumerate}
        \item $D^b_{\operatorname{coh}}(\mathcal{X}) = \langle G\rangle_n$,
        \item $\operatorname{Perf}(\mathcal{X}) \subseteq \langle G\rangle_n$.
    \end{enumerate}
\end{lemma}

\begin{proof}
    This follows immediately from \cite[Proposition 3.5]{Lank/Olander:2025} as approximation by compacts implies that `compact objects approximate $D^b_{\operatorname{coh}}(\mathcal{X})$' in the sense of loc.\ cit. A key point is that concentratedness of $\mathcal{X}$ ensures compact objects of $D_{\operatorname{qc}}(\mathcal{X})$ coincides with $\operatorname{Perf}(\mathcal{X})$. 
\end{proof}

\begin{theorem}
    \label{thm:stacky_bondal_vandenbergh}
    Let $\mathcal{X}$ be a concentrated separated Noetherian algebraic stack of finite Krull dimension satisfying approximation by compacts. 
    Then $\mathcal{X}$ is regular if, and only if, $\operatorname{Perf} (\mathcal{X})$ admits a strong generator.
\end{theorem}

\begin{proof}
    First, we assume that $\mathcal{X}$ is regular. By hypothesis, there exists a smooth surjective morphism from a regular affine scheme of finite Krull dimension to $\mathcal{X}$ (see \Cref{lem:stack_krull_dimension_finiteness_Noetherian}). It is known that for any regular Noetherian ring of finite Krull dimension its derived category admits a strong $\oplus$-generator which is a perfect complex (see \Cref{ex:regular_ring_strong_generator}). So, by \Cref{lem:descent_strong_oplus_bounded_coherent_stacks} and \Cref{cor:regularity_in_terms_of_perfectness}, it follows that 
    \begin{displaymath}
        D^b_{\operatorname{coh}}(\mathcal{X})=\operatorname{Perf}(\mathcal{X})=D_{\operatorname{qc}}(\mathcal{X})^c
    \end{displaymath}
    admits a strong generator.
    Lastly, we check the converse. In such a case, \Cref{lem:perf_strong_implies_dbcoh=perf} tells us $\operatorname{Perf}(\mathcal{X})$ admitting a strong generator implies $D^b_{\operatorname{coh}}(\mathcal{X})=\operatorname{Perf}(\mathcal{X})$. Consequently, \Cref{cor:regularity_in_terms_of_perfectness} ensures $\mathcal{X}$ is regular, which completes the proof.
\end{proof}

\begin{remark}
    \hfill
    \begin{enumerate}
        \item It follows from the proof, with $\mathcal{X}$ as in the theorem \emph{and} regular, that $D_{\operatorname{qc}}(\mathcal{X})$ admits a compact=perfect strong $\oplus$-generator.
        Also, note that approximation by compacts is only needed for one direction to apply \Cref{lem:perf_strong_implies_dbcoh=perf}.
        \item Ideally, one can show that $\operatorname{Perf}(\mathcal{X})$ strongly generated for a Noetherian algebraic stack automatically implies that the stack has finite Krull dimension (just as what happens in the scheme case). Unfortunately, we do not see how at the moment.
    \end{enumerate}
\end{remark}

The following is a reformulation of \Cref{thm:stacky_bondal_vandenbergh} more in line with the statement in \cite[Theorem 0.5]{Neeman:2021a}. 

\begin{corollary}
    \label{cor:stacky_neeman}
    Let $\mathcal{X}$ be as in \Cref{thm:stacky_bondal_vandenbergh}.
    Then $\operatorname{Perf}(\mathcal{X})$ admits a strong generator if, and only if, for every $x\in |\mathcal{X}|$ there exists a smooth morphism $f\colon \operatorname{Spec}(R_x) \to X$ with $R_x$ a (Noetherian) ring of finite global dimension whose image contains $x$.
\end{corollary}

\begin{proof}
     First suppose $\operatorname{Perf}(\mathcal{X})$ admits a strong generator, then \Cref{thm:stacky_bondal_vandenbergh} says $\mathcal{X}$ is regular.
     Let $\operatorname{Spec}(R) \to \mathcal{X}$ be a smooth surjective morphism from an affine scheme of finite Krull dimension (exists by e.g.\ \Cref{lem:stack_krull_dimension_finiteness_Noetherian}). Then $\operatorname{Spec}(R)$ must also be regular and, consequently, $R$ has finite global dimension. 
    
    Conversely, the condition implies any point $x\in|\mathcal{X}|$ is regular, by definition, since $\operatorname{Spec}( R_x)$ is regular. Hence, $\mathcal{X}$ is regular and \Cref{thm:stacky_bondal_vandenbergh} implies the desired claim.
\end{proof}

\section{Characterization via \texorpdfstring{$t$}{t}-structures}
\label{sec:regularity_t_structures}

In this last section, we show a version of \cite[Theorem 0.1]{Neeman:2022} for algebraic stacks. This requires a technical condition to be satisfied:

\begin{equation}
    \tag{$\textbf{Hypothesis }\star$}\label{hyp:stacky_approx}
    \begin{minipage}{10cm}
        An algebraic stack $\mathcal{X}$ satisfies \ref{hyp:stacky_approx} if it is $1$-Thomason and, for any quasi-compact open $\mathcal{U} \subseteq \mathcal{X}$ with complement $Z$, the standard $t$-structure on $D_{\operatorname{qc},Z}(\mathcal{X})$ is in the preferred equivalence class.
    \end{minipage}
\end{equation}

We need to make sure that the class of algebraic stacks satisfying \ref{hyp:stacky_approx} contains `new' cases outside the settinge of quasi-compact quasi-separated schemes (see \cite[Theorem 3.2(i)]{Neeman:2022}). Showing this requires some preliminary lemmas.

\begin{lemma}
    \label{lem:open_immersion_preferred_class}
    Let $\mathcal{X}$ be a quasi-compact quasi-separated algebraic stack. 
    If $j\colon \mathcal{U} \hookrightarrow \mathcal{X}$ is a quasi-compact open immersion and $\mathcal{X}$ satisfies \ref{hyp:stacky_approx}, then so does $\mathcal{U}$.
\end{lemma}

\begin{proof}
    An open substack of a 1-Thomason algebraic stack is again 1-Thomason (see e.g.\ \cite[Observation 5.6]{Neeman:2023b} or \cite[Lemma 4.4]{Lank:2026}). So, to show the other condition in \ref{hyp:stacky_approx}, let $\mathcal{V} \hookrightarrow \mathcal{U}$ be a quasi-compact open immersion with complement $Z$. 
    Composed with $j$, this yields a quasi-compact open immersion $\mathcal{V} \hookrightarrow \mathcal{X}$.
    Denote the closure of $Z$ in $|\mathcal{X}|$ by $\overline{Z}$. 
   
    Note that $\mathbf{R} j_{\ast} D^{\leq 0}_{\operatorname{qc},Z}(\mathcal{U})\subseteq D^{\leq r}_{\operatorname{qc},\overline{Z}}(\mathcal{X})$ for some integer $r\geq 0$.
    Indeed, the statement is smooth local, which allows one to reduce to the case of schemes where the result is known (see e.g.\ \cite[\href{https://stacks.math.columbia.edu/tag/08D5}{Tag's 08D5}, \href{https://stacks.math.columbia.edu/tag/01XJ}{01XJ}, \href{https://stacks.math.columbia.edu/tag/071L}{071L}]{stacks-project}). To show the desired claim, let $G$ be a compact generator for $D_{\operatorname{qc},\overline{Z}}(\mathcal{X})$ (which exists because $\mathcal{X}$ satisfies the $1$-Thomason condition). 
    Then $\mathbf{L}j^\ast G$ is a compact generator for $D_{\operatorname{qc},Z}(\mathcal{U})$. Indeed, $\mathbf{R}j_\ast$ is conservative because it is quasi-affine \cite[Corollary 2.8]{Hall/Rydh:2017}, and so $\mathbf{L}j^\ast G$ is a compact generator (see e.g.\ \cite[Corollary 3.7]{Lank:2026}). As $\mathcal{X}$ satisfies \ref{hyp:stacky_approx}, there exists an $A>0$ such that 
    \begin{displaymath}
        \operatorname{Coprod}(\{ G[s]\mid s\geq 0 \})[A] \subseteq D^{\leq 0}_{\operatorname{qc},\overline{Z}}(\mathcal{X}) \subseteq \operatorname{Coprod}(\{ G[s]\mid s\geq 0 \})[-A].
    \end{displaymath}
    Since $j$ is flat, we know that 
    \begin{displaymath}
        \operatorname{Coprod}(\{ \mathbf{L} j^\ast G[s] \mid s\geq 0 \})[A] \subseteq D^{\leq 0}_{\operatorname{qc},Z}(\mathcal{U})
    \end{displaymath}
    Choose $E\in D^{\leq 0}_{\operatorname{qc},Z}(\mathcal{U})$. Then 
    \begin{displaymath}
        \mathbf{R} j_{\ast} E \in D^{\leq r}_{\operatorname{qc},\overline{Z}}(\mathcal{X})\subseteq \operatorname{Coprod}(\{ G[s]\mid s\geq 0 \})[-A-r]
    \end{displaymath}
    and so, 
    \begin{displaymath}
        E \cong \mathbf{L} j^\ast \mathbf{R} j_{\ast} E \in \operatorname{Coprod}(\{ \mathbf{L} j^\ast G[s] \mid  s\geq 0 \})[-A-r].
    \end{displaymath} 
    Consequently, we have
    \begin{displaymath}
        \operatorname{Coprod}(\{ \mathbf{L} j^\ast G[s] \mid s\geq 0 \})[A] \subseteq D^{\leq 0}_{\operatorname{qc},Z}(\mathcal{U}) \subseteq \operatorname{Coprod}(\{ \mathbf{L} j^\ast G[s] \mid  s\geq 0 \})[-A-r],
    \end{displaymath}
    which completes the proof.
\end{proof}

Recall that a quasi-compact morphism $f\colon\mathcal{X}\to\mathcal{Y}$ of algebraic stacks is \textbf{cohomologically affine} if the pushforward of $f$ is exact on quasi-coherent sheaves \cite[Definition\ 3.1]{Alper:2013}.

\begin{lemma}
    \label{lem:finite_faithfully_flat_affine_cover_preferred_class}
    Let $\mathcal{X}$ be a concentrated algebraic stack. Suppose there exists a finite flat surjective morphism $f\colon V \to \mathcal{X}$ of finite presentation from an affine scheme $V$. Then the natural morphism $E \to \mathbf{R}f_{\ast} \mathbf{L}f^\ast E$ is a split monomorphism for all $E\in D_{\operatorname{qc}}(\mathcal{X})$. Moreover, $\mathcal{X}$ satisfies \ref{hyp:stacky_approx}.
\end{lemma}

\begin{proof}
    By \cite[Theorem C]{Hall/Rydh:2017}, we know that $\mathcal{X}$ is $1$-Thomason. Moreover, \cite[Theorem C]{Hall/Rydh:2015} ensurs that $\mathcal{X}$ has linearly reductive stabilizers at closed points. Note that $\mathcal{X}$ has affine diagonal because it admits a finite flat surjective morphism from an affine scheme. Furthermore, $\mathcal{X}$ admits a coarse moduli space $\pi\colon\mathcal{X}\to X$ where $X$ is an affine scheme. See e.g.\ \cite[Section\ 6.2]{Olsson:2016} for a proof of this fact. From \cite[Corollary\ 5.8]{Rydh:2023}, we know that the morphism $\pi\colon\mathcal{X}\to X$ is a good moduli space. So, in particular, $\pi$ is cohomologically affine. Hence, $\mathcal{X}$ itself is cohomologically affine.

    Let us list some more facts we shall use without mention:
    \begin{itemize}
        \item First, we claim that the natural morphism $\mathcal{O}_\mathcal{X} \to \mathbf{R}f_{\ast} \mathcal{O}_V$ splits.
        In particular, using the projection formula, it follows that the natural morphism $E \to \mathbf{R}f_{\ast} \mathbf{L}f^\ast E$ is a split monomorphism for all $E\in D_{\operatorname{qc}}(\mathcal{X})$ (e.g.\ tensor with $E$ and use \cite[Lemma 5.15]{DeDeyn/Lank/ManaliRahul:2024b}). Note that $\mathcal{O}_\mathcal{X} \to \mathbf{R}f_{\ast} \mathcal{O}_V$ is injective as $f$ is schematically dominant (e.g.\ see above \cite[Definition 3.4]{Rydh:2016a}). 
        Since $f$ is finite flat and of finite presentation, we have the following short exact sequence of coherent sheaves on $\mathcal{X}$,
        \begin{displaymath}
            0\rightarrow\mathcal{O}_\mathcal{X}\xrightarrow{ntrl.} \mathbf{R}f_{\ast} \mathcal{O}_V\rightarrow\mathcal{V}\rightarrow 0.
        \end{displaymath}
        It suffices to show $\mathcal{V}$ is a locally free sheaf. Indeed, as $\mathcal{X}$ is cohomologically affine, we have $\operatorname{Ext}^{1}(\mathcal{V},\mathcal{O}_{\mathcal{X}})=H^{1}(\mathcal{X},\mathcal{V}^\vee)=0$, which gives the desired claim. So, we are left to check that $\mathcal{V}$ is a locally free sheaf. Choose a smooth surjective morphism $s\colon U \to \mathcal{X}$ from a scheme. By descent, $\mathcal{V}$ being a locally free sheaf means that $s^\ast \mathcal{V}$ is locally free on $U$ (see e.g.\ \cite[Definition 2.1]{Rydh:2016}). Since $f$ is representable by schemes, we have the induced morphism $f^\prime\colon \mathcal{Y} \times_\mathcal{X} U \to U$ is a finite flat surjective morphism of (Noetherian) schemes. We then have a short exact sequence
        \begin{displaymath}
            0\rightarrow\mathcal{O}_U\xrightarrow{ntrl.} \mathbf{R}f^\prime_{\ast} \mathcal{O}_{\mathcal{Y} \times_\mathcal{X} U} \rightarrow s^\ast\mathcal{V}\rightarrow 0.
        \end{displaymath}
        From $f^\prime$ being affine, we can reduce to the case where $f^\prime$ is a morphism of affine schemes. If we can show $\mathcal{O}_U\xrightarrow{ntrl.} \mathbf{R}f^\prime_{\ast} \mathcal{O}_{\mathcal{Y} \times_\mathcal{X} U}$ splits, then $s^\ast\mathcal{V}$ is a direct summand of a locally free sheaf, and so, the claim follows. Now we know that $\mathbf{R}f^\prime_{\ast} \mathcal{O}_{\mathcal{Y} \times_\mathcal{X} U }$ is a locally free sheaf. Hence, for each $p\in U$, we know that $(\mathbf{R}f^\prime_{\ast} \mathcal{O}_{\mathcal{Y} \times_\mathcal{X} U})_p$ contains $\mathcal{O}_{U,p}$ as a direct summand. In particular, we have that $\mathcal{O}_{U,p}\in \langle (\mathbf{R}f^\prime_{\ast} \mathcal{O}_{\mathcal{Y} \times_\mathcal{X} U})_p \rangle_1$ for all $p\in U$. This tells us, by \cite[Corollary 3.4]{Letz:2021}, that $\mathcal{O}_U \in \langle \mathbf{R}f^\prime_{\ast} \mathcal{O}_{\mathcal{Y} \times_\mathcal{X} U} \rangle_1$. Then $\mathcal{O}_U$ is a direct summand of an object of the form $ \mathbf{R}f^\prime_{\ast} E$ for some $E\in D^b_{\operatorname{coh}}(\mathcal{Y} \times_\mathcal{X} U)$. However, adjunction implies that $\mathcal{O}_U\xrightarrow{ntrl.} \mathbf{R}f^\prime_{\ast} \mathcal{O}_{\mathcal{Y} \times_\mathcal{X} U}$ splits as desired (see e.g.\ \cite[Lemma 5.15]{DeDeyn/Lank/ManaliRahul:2024b}).
        \item For any quasi-compact open immersion $j\colon \mathcal{U} \to \mathcal{X}$ with complement $Z$, there exists a $G\in D_{\operatorname{qc}}(\mathcal{X})^c$ which is a compact generator for $D_{\operatorname{qc},Z}(\mathcal{X})$. This follows from \cite[Lemma 4.10]{Hall/Rydh:2017} and our hypothesis that $\mathcal{X}$ satisfies the $1$-Thomason condition. Moreover, $\mathbf{L}f^\ast G$ is a compact generator for $D_{\operatorname{qc},f^{-1}(Z)} (V)$.
        Again, $\mathbf{R}f_\ast$ is conservative \cite[Corollary 2.8]{Hall/Rydh:2017}, and so $\mathbf{L}f^\ast$ preserves compact generators (see e.g.\ \cite[Corollary 3.7]{Lank:2026});
        it can be checked $\mathbf{R}f_\ast D_{\operatorname{qc},f^{-1}(Z)}(V)\subseteq D_{\operatorname{qc},Z} (\mathcal{X})$ by passing smooth locally and reducing to schemes where the fact is known (see e.g.\ \cite[Remark 23.46(2)]{Gortz/Wedhorn:2023}).
    \end{itemize}

    Now, let us prove the desired claim. Observe that $\mathbf{R}f_\ast \mathbf{L}f^\ast G$ is perfect whose support is contained in $Z$. Moreover, $\mathbf{R}f_\ast \mathbf{L}f^\ast G$ contains $G$ as a direct summand. Hence, $\mathbf{R}f_\ast \mathbf{L}f^\ast G$ is a compact object generating $D_{\operatorname{qc},Z}(\mathcal{X})$.
    Furthermore, if needed, we can shift $G$ to ensure that $\mathbf{R}f_{\ast} \mathbf{L}f^\ast G\in D^{\leq 0}_{\operatorname{qc},Z}(\mathcal{X})$. Then we know that,
    \begin{displaymath}
        \operatorname{Coprod}(\{\mathbf{R}f_{\ast} \mathbf{L}f^\ast G [i]\mid  i\geq 0\})\subseteq D^{\leq 0}_{\operatorname{qc},Z}(\mathcal{X}).
    \end{displaymath}
    To show the other required containment, note that, as affine schemes satisfy \ref{hyp:stacky_approx}, we can choose an $A>0$ with 
    \begin{displaymath}
        D^{\leq 0}_{\operatorname{qc},f^{-1}(Z)}(V)\subseteq \operatorname{Coprod}(\{\mathbf{L}f^\ast G [i]\mid s\geq 0 \})[A].
    \end{displaymath}
    Now, observe that, since $f$ is flat, $\mathbf{L}f^\ast D^{\leq 0}_{\operatorname{qc},Z}(\mathcal{X})\subseteq D^{\leq 0}_{\operatorname{qc},f^{-1}(Z)}(V) $ and using the above splitting that every object of $D^{\leq 0}_{\operatorname{qc},Z}(\mathcal{X})$ is a direct summand of an object belonging to $\mathbf{R}f_{\ast} \mathbf{L}f^\ast D^{\leq 0}_{\operatorname{qc},Z}(\mathcal{X})$.
    Consequently, as aisles are closed under summands,
    \begin{displaymath}
        D^{\leq 0}_{\operatorname{qc},Z}(\mathcal{X})\subseteq \operatorname{Coprod}(\{\mathbf{R}f_{\ast} \mathbf{L}f^\ast G [i]\mid i\geq 0 \})[A]
    \end{displaymath}
    which completes the proof.
\end{proof}

\begin{remark}
    The splitting in the statement of \Cref{lem:finite_faithfully_flat_affine_cover_preferred_class} is equivalent to \cite[Lemma\ A.1]{Hall:2016}; in fact, loc.\ cit.\ has a slight error in the proof, which the above fixes at least for the concentrated case.
\end{remark}

\begin{lemma}
    \label{lem:BNP_weak_version_for_glueing}
    Consider a recollement of triangulated categories 
    \begin{displaymath}
        \begin{tikzcd}[ampersand replacement=\&]
            {\mathcal{D}_Z} \&\& {\mathcal{D}} \&\& {\mathcal{D}_U}\rlap{ .}
            \arrow["{i_\ast}" description, from=1-1, to=1-3]
            \arrow["{i^!}", bend right = -30 pt, from=1-3, to=1-1]
            \arrow["{i^\ast}"', bend right = 30 pt, from=1-3, to=1-1]
            \arrow["{j^\ast}" description, from=1-3, to=1-5]
            \arrow["{j_!}"', bend right = 30 pt, from=1-5, to=1-3]
            \arrow["{j_\ast}", bend right = -30 pt, from=1-5, to=1-3]
        \end{tikzcd}
    \end{displaymath}
    Suppose $\mathcal{D}_U$ admits a compact generator $G_U$, and $\mathcal{D}$ admits a compact generator $G^\prime$ with $\operatorname{Hom}(G^\prime [-n],G^\prime)=0$ for $n\gg 0$ and $j^\ast G^\prime\in \mathcal{D}_U^{\leq m}$ for some $m>0$, where $\mathcal{D}_U^{\leq 0}$ is the aisle of the $t$-structure on $\mathcal{D}_U$ compactly generated by $G_U$. 
    Then there exists a compact generator $G\in \mathcal{D}$, which can be taken to be $G^\prime[m] \oplus j_! G_U$,  such that the $t$-structure on $\mathcal{D}$ compactly generated by $G$ is the gluing of the $t$-structure on $\mathcal{D}_Z$ compactly generated by $i^\ast G$ and the $t$-structure on $\mathcal{D}_U$ compactly generated by $G_U$. 
    In particular, the glued $t$-structure is in the preferred equivalence class.
\end{lemma}

\begin{proof}
    This is argued verbatim to \cite[Lemma 3.11]{Burke/Neeman/Pauwels:2023} where in loc.\ cit.\ one can replace the condition $\mathcal{D}_U$ `weakly approximable' by our hypothesis $j^\ast G^\prime\in\mathcal{D}_U^{\leq m}$ for some $m>0$.
\end{proof}

   Consider a cartesian diagram of quasi-compact quasi-separated algebraic stacks
    \begin{equation}
        \label{eq:MV_square}
        \begin{tikzcd}
            f^{-1}(\mathcal{U})\arrow[r,"i", hook]\arrow[d,"g"] & \mathcal{Y}\arrow[d,"f"]\\
            \mathcal{U}\arrow[r,"j", hook] & \mathcal{X}\rlap{ ,}
        \end{tikzcd}
    \end{equation}
    where $j$ is a quasi-compact open immersion.
    Recall, from \cite[Definition 1.2]{Hall/Rydh:2023}, that 
    \eqref{eq:MV_square} is called a \textbf{flat Mayer--Vietoris square} when 
    \begin{itemize}
        \item $f$ is flat and for every morphism of algebraic stacks $\mathcal{W} \to \mathcal{X}$ with image disjoint from $\mathcal{U}$, the induced morphism $\mathcal{X}^\prime \times_{\mathcal{X}} \mathcal{W} \to \mathcal{W}$ is an isomorphism.
    \end{itemize}
    When $f$ is \'{e}tale, this is also known as an `\'{e}tale neighborhood' (see \cite[Lemma 2.1]{Rydh:2011}).

\begin{lemma}
    \label{lem:etale_nbhd_preferred_class}
    Consider a flat Mayer--Vietoris square as in \eqref{eq:MV_square} where $f$ is concentrated. If $\mathcal{U}$ and $\mathcal{Y}$ satisfy \ref{hyp:stacky_approx}, then so does $\mathcal{X}$.
\end{lemma}

\begin{proof}
    That $\mathcal{X}$ satisfies the $1$-Thomason condition follows by \cite[Example 6.2 \& Proposition 6.8]{Hall/Rydh:2017}. 
    So, we only need to show that $\mathcal{X}$ satisfies the rest of \ref{hyp:stacky_approx}. Set $Z:=|\mathcal{X}|\setminus |\mathcal{U}|$. Let $\mathcal{V}\to \mathcal{X}$ be a quasi-compact open immersion. Denote by $W:=|\mathcal{X}|\setminus |\mathcal{V}|$. By \cite[Corollary 3.3]{Lank:2026}, there exists a recollement
    \begin{displaymath}
        \begin{tikzcd}
            {D_{\operatorname{qc},|\mathcal{U}|\cap W}(\mathcal{U})} && {D_{\operatorname{qc},W}(\mathcal{X})} && {D_{\operatorname{qc},Z\cap W}(\mathcal{X})}
            \arrow["{\mathbf{R}j_\ast}"{description}, from=1-1, to=1-3]
            \arrow["{\mathbf{L}j^\ast}"', bend right =20pt, from=1-3, to=1-1]
            \arrow["{j^!}", bend right =-20pt, from=1-3, to=1-1]
            \arrow["{i^!}"{description}, from=1-3, to=1-5]
            \arrow["{i_\ast}"', bend right =20pt, from=1-5, to=1-3]
            \arrow["{i^\ast}", bend right =-20pt, from=1-5, to=1-3]
        \end{tikzcd}
    \end{displaymath}
    where $i_\ast$ is the natural inclusion (loc.\ cit.\ uses that $\mathcal{U}$ satisfies the Thomason condition). 

    To start, we make the observation that there is an $a\geq 0$ such that $i^! D^{\leq 0}_{\operatorname{qc},W}(\mathcal{X})\subseteq D^{\leq a+1}_{\operatorname{qc},Z\cap W}(\mathcal{X})$. Let $E\in D^{\leq 0}_{\operatorname{qc},W}(\mathcal{X})$. Since $j$ is a flat quasi-affine morphism, we know that $\mathbf{L}j^\ast E\in D_{\operatorname{qc},|\mathcal{U}|\cap W}^{\leq 0}(\mathcal{U})$ and $\mathbf{R}j_\ast D^{\leq 0}_{\operatorname{qc},|\mathcal{U}|\cap W}(\mathcal{U})\subseteq D^{\leq a}_{\operatorname{qc},W}(\mathcal{X})$ for some $a\geq 0$. 
    Indeed, flatness ensures $\mathbf{L}j^\ast$ is $t$-exact with respect to the standard $t$-structures, whereas open immersions being concentrated ensures the existence of such an $a$.
    The recollement gives us a distinguished triangle 
    \begin{equation}\label{eq:rec}
        i_\ast i^! E \to E \to \mathbf{R}j_\ast \mathbf{L}j^\ast E \to i_\ast i^! E [1].
    \end{equation}
    Note that $\mathbf{R}j_\ast \mathbf{L}j^\ast E  \in D^{\leq a}_{\operatorname{qc},W}(\mathcal{X})$ and
    \begin{displaymath}
        E\in D^{\leq 0}_{\operatorname{qc},W}(\mathcal{X}) \subseteq D^{\leq 0}_{\operatorname{qc},W}(\mathcal{X})[-a] = D^{\leq a}_{\operatorname{qc},W}(\mathcal{X}).
    \end{displaymath}
    So, we see from \eqref{eq:rec} that $i_\ast i^! E [1] \in D^{\leq a}_{\operatorname{qc},W}(\mathcal{X})$. Hence, it follows that $i^! E \in D^{\leq a+1}_{\operatorname{qc},Z\cap W}(\mathcal{X})$ because $i_\ast$ is really just the inclusion (although, another way to see this is to use $t$-exactness and conservativeness) with respect to the standard $t$-structure, which shows the desired observation.

    Now, we return to proving the desired claim. As $\mathcal{X}$ is $1$-Thomason, we know that $D_{\operatorname{qc},W}(\mathcal{X})$ is singly compactly generated, and so we can find $P\in D_{\operatorname{qc},W}(\mathcal{X})^c$ such that $\mathbf{L}j^\ast P$ compactly generates $D_{\operatorname{qc},|\mathcal{U}|\cap W}(\mathcal{U})$. Indeed, $\mathbf{R}j_\ast$ is conservative, and so $\mathbf{L}j^\ast P$ is a compact generator whenever $P$ is such for $D_{\operatorname{qc},W}(\mathcal{X})$ (see e.g.\ \cite[Corollary 3.7]{Lank:2026}). Moreover, by 
    \cite[Proposition 4.2]{Hall/Rydh:2023}, $\mathbf{R}f_\ast$ restricts to an equivalence $D_{\operatorname{qc},f^{-1}(Z\cap W)}(\mathcal{Y})\to D_{\operatorname{qc},Z\cap W}(\mathcal{X})$ (which requires $f$ to be concentrated). In particular, this equivalence is $t$-exact with respect to the standard $t$-structures. As $\mathcal{Y}$ is $1$-Thomason, it follows that $D_{\operatorname{qc},f^{-1}(Z\cap W)}(\mathcal{Y})$, and hence $D_{\operatorname{qc},Z\cap W}(\mathcal{X})$, is compactly generated. Let $P_Z$ be a compact generator for $ D_{\operatorname{qc},Z\cap W}(\mathcal{X})$. 
    As $i^!$ preserves coproducts (as it admits a right adjoint) it follows formally that $i_\ast$ preserves compacts.
    So $i_\ast P_Z \in D_{\operatorname{qc},W}(\mathcal{X})^c$.

    We are in the setting where $\operatorname{Hom}(P  [-n], P ) = 0$ for $n\gg 0$.
    Moreover, after shifting if needed, we can impose $P \oplus i_\ast P_Z\in D^{\leq 0}_{\operatorname{qc}}(\mathcal{X})$.
    Hence, by the observation established above, we have $i^! P \in D^{\leq a+1}_{\operatorname{qc},Z\cap W}(\mathcal{X})$.
    Consequently, by \Cref{lem:BNP_weak_version_for_glueing}, the glued $t$-structure on $D_{\operatorname{qc},W}(\mathcal{X})$, of the aisles on respectively $D_{\operatorname{qc},|\mathcal{U}|\cap W}(\mathcal{U})$ and $D_{\operatorname{qc},Z\cap W}(\mathcal{X})$ compactly generated by $\mathbf{L}j^\ast P$ and $P_Z$, is compactly generated by $P [a+1]\oplus i_\ast P_Z$.

    We claim that there is an $N\geq 0$ such that 
    \begin{displaymath}
        D_{\operatorname{qc},W}^{\leq -N}(\mathcal{X})\subseteq\operatorname{Coprod}(\{ (P\oplus i_\ast P_Z) [s]\ |\ s\geq 0\})\subseteq D_{\operatorname{qc},Z}^{\leq N}(\mathcal{X})
    \end{displaymath}
    which shows our desired claim as $(P\oplus i_\ast P_Z)$ is also a compact generator.

    First, as $P \oplus i_\ast P_Z\in D^{\leq 0}_{\operatorname{qc}}(\mathcal{X})$, it follows that
    \begin{displaymath}
        \operatorname{Coprod}(\{ (P \oplus i_\ast P_Z) [s]\ |\ s\geq 0\})\subseteq D_{\operatorname{qc},Z}^{\leq t}(\mathcal{X}) = D_{\operatorname{qc},Z}^{\leq 0}(\mathcal{X})[-t]
    \end{displaymath}
    for all $t\geq 0$. 
    Thus, it remains to show there is an $N\geq 0$ such that $D_{\operatorname{qc},W}^{\leq -N}(\mathcal{X})\subseteq\operatorname{Coprod}(\{ (P \oplus i_\ast P_Z) [s]\ |\ s\geq 0\})$. 
    This can be done by showing there is an $N\geq 0$ such that
    \begin{displaymath}
        \begin{aligned}
            \mathbf{L}j^\ast D^{\leq -N}_{\operatorname{qc},W}(\mathcal{X})&\in \operatorname{Coprod} (\{ \mathbf{L}j^\ast P[s]\  |\ s\geq 0\})\\ i^! D^{\leq -N}_{\operatorname{qc},W}(\mathcal{X})&\in \operatorname{Coprod}(\{ P_Z [s]\ |\ s\geq 0\}).
        \end{aligned}
    \end{displaymath}
    Indeed, by of the glued $t$-structure, this implies that
    \begin{displaymath}
        \begin{aligned}
            D^{\leq -N}_{\operatorname{qc},W}(\mathcal{X}) 
            &\subseteq \operatorname{Coprod} (\{ (P[a+1]\oplus i_\ast P_Z) [s]\ |\ s\geq 0\}) 
            \\&\subseteq \operatorname{Coprod} (\{ (P\oplus i_\ast P_Z) [s]\ |\ s\geq 0\}).
        \end{aligned}
    \end{displaymath}
    Towards that end, using the hypothesis on $\mathcal{U}$ and $\mathcal{Y}$, for all $m\gg 0$ we have
    \begin{displaymath}
        \begin{aligned}
            D_{\operatorname{qc},|\mathcal{U}|\cap W}^{\leq -m}  (\mathcal{U}) &\subseteq \operatorname{Coprod}(\{\mathbf{L}j^\ast P [s]\ |\ s\geq 0\})       \quad\text{and}     \\
            D_{\operatorname{qc},Z\cap W}^{\leq -m}(\mathcal{X}) &\subseteq\operatorname{Coprod}(\{ P_Z [s]\ |\ s\geq 0\}).
        \end{aligned}
    \end{displaymath}
    Consequently, it follows that by taking $N:=m+a + 1 \geq 0$, that
    \begin{displaymath}
        \begin{aligned}
            \mathbf{L}j^\ast D^{\leq -N}_{\operatorname{qc},W}(\mathcal{X})&\subseteq D_{\operatorname{qc},|\mathcal{U}|\cap W}^{\leq -N} (\mathcal{U})  \\
            &\subseteq D_{\operatorname{qc},|\mathcal{U}|\cap W}^{\leq -m} (\mathcal{U}) \\&\subseteq \operatorname{Coprod}(\{\mathbf{L}j^\ast P [s]\ |\ s\geq 0\}),
        \end{aligned}
    \end{displaymath}
    and that 
    \begin{displaymath}
        \begin{aligned}
            i^! D^{\leq -N}_{\operatorname{qc},W}(\mathcal{X})
            &\subseteq D^{\leq -m}_{\operatorname{qc},Z\cap W}(\mathcal{X})
            \\&\subseteq \operatorname{Coprod}(\{ P_Z [s]\ |\ s\geq 0\})
        \end{aligned}
    \end{displaymath}
    showing the desired claim and completing the proof.
\end{proof}

\begin{proposition}
    \label{prop:preferred_equivalence_classes}
    Let $\mathcal{S}$ be a concentrated algebraic stack that either    
    \begin{enumerate}
        \item \label{item:stacks_preferred1} has quasi-finite and separated diagonal or
        \item \label{item:stacks_preferred2} is a Deligne--Mumford stack of characteristic zero.
    \end{enumerate}
    Then $\mathcal{S}$ satisfies \ref{hyp:stacky_approx}.
\end{proposition}

\begin{proof}
    First, we prove case \eqref{item:stacks_preferred1}. So, assume $\mathcal{S}$ is a concentrated algebraic stack with quasi-finite and separated diagonal. Define $\mathbb{E}$ to be the strictly full $2$-subcategory of algebraic stacks over $\mathcal{S}$ consisting of algebraic stacks whose structure morphism $\mathcal{X} \to \mathcal{S}$ is representable by algebraic spaces, separated, finitely presented and quasi-finite flat. 

    Observe the following facts concerning $\mathbb{E}$:
    \begin{itemize}
        \item The source of every object in $\mathbb{E}$ is quasi-compact quasi-separated as any finitely presented morphism of stacks is quasi-compact quasi-separated by definition.
        \item Every morphism in $\mathbb{E}$ is representable by algebraic spaces (see e.g.\ \cite[Lemma 6.7]{DeDeyn/Lank/ManaliRahul:2025}), and so each morphism in $\mathbb{E}$ is concentrated by \cite[Lemma 2.5(3)]{Hall/Rydh:2017}. 
        In particular, as $\mathcal{S}$ is concentrated, every source of an object in $\mathbb{E}$ is concentrated.
    \end{itemize}
    
    Define $\mathbb{D}$ be the strictly full $2$-subcategory of $\mathbb{E}$ consisting of objects that satisfy \ref{hyp:stacky_approx}. 
    We invoke \cite[Theorem E]{Hall/Rydh:2018}\footnote{There is a typo in loc.\ cit.\ known to experts, but we reminder the reader: it suffices to only check (I2) for morphisms that are additionally \textit{flat}.} to show $\mathbb{E}=\mathbb{D}$. To this end, we need to verify the following:
    \begin{itemize}
        \item \label{item:preferred_cover1} if $(\mathcal{U} \to \mathcal{X})\in\mathbb{E}$ is an open immersion and $\mathcal{X}\in\mathbb{D}$, then $\mathcal{U}\in\mathbb{D}$,
        \item \label{item:preferred_cover2} if $(V \to \mathcal{X})\in \mathbb{E}$ is finite, flat and surjective with affine source, then $\mathcal{X}\in\mathbb{D}$, and
        \item \label{item:preferred_cover3} if $(\mathcal{U} \xrightarrow{i} \mathcal{X})$, $(\mathcal{Y} \xrightarrow{f} \mathcal{X})\in\mathbb{E}$, where $i$ is an open immersion and $f$ is \'{e}tale which form an \'{e}tale neighborhood, then $\mathcal{X}\in \mathbb{D}$ whenever $\mathcal{U}$, $\mathcal{Y}\in\mathbb{D}$.
    \end{itemize}
    As these are exactly \Cref{lem:open_immersion_preferred_class,lem:finite_faithfully_flat_affine_cover_preferred_class,lem:etale_nbhd_preferred_class}, the desired claim follows.

    To see \eqref{item:stacks_preferred2}, one can adapt the prior argument.
    Instead, let $\mathbb{E}$ denote the strictly full $2$-subcategory of algebraic stacks over $\mathcal{S}$ consisting of algebraic stacks whose structure morphism $\mathcal{X} \to \mathcal{S}$ is finitely presented, \'{e}tale and has separated diagonal. 
    Again, the source of any object in $\mathbb{E}$ is quasi-compact quasi-separated and moreover Deligne--Mumford (this follows by the definition of \'{e}tale morphisms of stacks \cite[\href{https://stacks.math.columbia.edu/tag/0CIL}{Tag 0CIL}]{stacks-project}) of characteristic zero.
    Moreover, any quasi-compact quasi-separated Deligne--Mumford stack of characteristic zero is concentrated by \cite[Theorem C]{Hall/Rydh:2015} as they have finite---and so affine---stabalizers (and so also any morphism between such by \cite[Lemma 2.5(4)]{Hall/Rydh:2017}).
    Thus, one can argue in a similar fashion as \eqref{item:stacks_preferred1}, again invoking \cite[Theorem E]{Hall/Rydh:2018}.
\end{proof}

Given the legwork above, we end with the following analog of \cite[Theorem 0.1]{Neeman:2022}.

\begin{theorem}
    \label{thm:stacky_neeman_bounded_t_structure}
    Let $\mathcal{X}$ be an algebraic stack as in \Cref{prop:preferred_equivalence_classes} that is additionally Noetherian and of finite Krull dimension.
    Suppose $Z$ is a closed subset of $\mathcal{X}$.
    Then $Z\subseteq\operatorname{reg}(\mathcal{X})$ if, and only if, $\operatorname{Perf}_Z (\mathcal{X})$ admits a bounded $t$-structure.
\end{theorem} 

\begin{proof}
    The forward direction follows from \Cref{prop:closed_subset_regular_locus}, whereas the converse can be shown essentially verbatim as in \cite[$\S 3$]{Neeman:2022} making use of the following observations:
    \begin{itemize}
        \item The necessary parts of \cite[Theorem 3.2]{Neeman:2022} hold in this setting. 
        Indeed, (i) and (ii) of loc.\ cit.\ follow as $\mathcal{X}$ is concentrated and satisfies the $1$-Thomason condition, the latter and  (iii) are exactly the content of \Cref{prop:preferred_equivalence_classes}.
        Furthermore, (iv) is not actually needed in \cite[$\S 3$]{Neeman:2022}---as explicitly noted in loc.\ cit. 
        \item The analogue of \cite[Lemma 3.5]{Neeman:2022} follows from \Cref{lem:hall_pullback_sheaf_hom}. Indeed, by the latter the question is smooth local and so one reduces to \cite[Lemma 3.5]{Neeman:2022}.
        \item The proof of \cite[Lemma 3.4]{Neeman:2022} only uses the above ingredients. 
        \item That the theorem follows from loc.\ cit.\ requires \cite[Theorem 3.3]{Neeman:2022}; in this setting this is simply requiring that $\mathcal{X}$ satisfies approximation by compacts.
        This holds by \cite[Corollary 5.4 \& Corollary 5.5]{Hall/Lamarche/Lank/Peng:2025}. \qedhere 
    \end{itemize}
\end{proof}

\bibliographystyle{alpha}
\bibliography{mainbib}

@misc{stacks-project,
  shorthand    = {Stacks},
  author       = {The {Stacks Project Authors}},
  title        = {\textit{Stacks Project}},
  howpublished = {\url{https://stacks.math.columbia.edu}},
  year         = {2025},
}

@Article{EGAIV4:1967,
  Author = {Grothendieck, A.},
  Title = {{\'E}l{\'e}ments de g{\'e}om{\'e}trie alg{\'e}brique. {IV}: {\'E}tude locale des sch{\'e}mas et des morphismes de sch{\'e}mas ({Quatri{\`e}me} partie). {R{\'e}dig{\'e}} avec la colloboration de {J}. {Dieudonn{\'e}}},
  FJournal = {Publications Math{\'e}matiques},
  Journal = {Publ. Math., Inst. Hautes {\'E}tud. Sci.},
  ISSN = {0073-8301},
  Volume = {32},
  Pages = {1--361},
  Year = {1967},
  Language = {French},
  zbMATH = {3245973},
  Zbl = {0153.22301}
}

@article {Lipman/Neeman:2007,
  AUTHOR = {Lipman, Joseph and Neeman, Amnon},
  TITLE = {Quasi-perfect scheme-maps and boundedness of the twisted
        inverse image functor},
  JOURNAL = {Illinois J. Math.},
  FJOURNAL = {Illinois Journal of Mathematics},
  VOLUME = {51},
  YEAR = {2007},
  NUMBER = {1},
  PAGES = {209--236},
  ISSN = {0019-2082,1945-6581},
  MRCLASS = {14A15},
  MRNUMBER = {2346195},
  MRREVIEWER = {Stefan\ Schr\"{o}er},
  URL = {http://projecteuclid.org/euclid.ijm/1258735333},
}

@article{Neeman:1996,
  author = {Neeman, Amnon},
  title = {The {Grothendieck} duality theorem via {Bousfield}'s techniques and {Brown} representability},
  fjournal = {Journal of the American Mathematical Society},
  journal = {J. Am. Math. Soc.},
  issn = {0894-0347},
  volume = {9},
  number = {1},
  pages = {205--236},
  year = {1996},
  language = {English},
  doi = {10.1090/S0894-0347-96-00174-9},
  zbMATH = {868352},
  Zbl = {0864.14008}
}

@article{Kelly:1965,
  author = {Kelly, G. M.},
  title = {Chain maps inducing zero homology maps},
  fjournal = {Proceedings of the Cambridge Philosophical Society},
  journal = {Proc. Camb. Philos. Soc.},
  issn = {0008-1981},
  volume = {61},
  pages = {847--854},
  year = {1965},
  language = {English},
  zbMATH = {3236072},
  Zbl = {0146.25106}
}

@article{Street:1973,
  author = {Street, Ross},
  title = {Homotopy classification of filtered complexes},
  fjournal = {Journal of the Australian Mathematical Society},
  journal = {J. Aust. Math. Soc.},
  issn = {1446-7887},
  volume = {15},
  pages = {298--318},
  year = {1973},
  language = {English},
  doi = {10.1017/S1446788700013227},
  zbMATH = {3422604},
  Zbl = {0268.18016}
}

@article{Christensen:1998,
  author = {Christensen, J. Daniel},
  title = {Ideals in triangulated categories: {Phantoms}, ghosts and skeleta},
  fjournal = {Advances in Mathematics},
  journal = {Adv. Math.},
  issn = {0001-8708},
  volume = {136},
  number = {2},
  pages = {284--339},
  year = {1998},
  language = {English},
  doi = {10.1006/aima.1998.1735},
  url = {hdl.handle.net/1721.1/42695},
  zbMATH = {1215519},
  Zbl = {0928.55010}
}

@article{Rydh:2016a,
  author = {Rydh, David},
  title = {Approximation of sheaves on algebraic stacks},
  fjournal = {IMRN. International Mathematics Research Notices},
  journal = {Int. Math. Res. Not.},
  issn = {1073-7928},
  volume = {2016},
  number = {3},
  pages = {717--737},
  year = {2016},
  language = {English},
  doi = {10.1093/imrn/rnv142},
  zbMATH = {6560845},
  Zbl = {1353.14003}
}

@misc{Hall/Lamarche/Lank/Peng:2025,
  title={Compact approximation and descent for algebraic stacks}, 
  author={Jack Hall and Alicia Lamarche and Pat Lank and Fei Peng},
  year={2025},
  url={https://arxiv.org/abs/2504.21125}, 
  eprint={2504.21125},
  archivePrefix={arXiv},
  howpublished	= {\href{https://arxiv.org/abs/arXiv:2504.21125}{arXiv:2504.21125}},
  publisher     = {arXiv},
}

@Book{Gortz/Wedhorn:2023,
  Author = {G{\"o}rtz, Ulrich and Wedhorn, Torsten},
  Title = {Algebraic geometry {II}: cohomology of schemes. {With} examples and exercises},
  FSeries = {Springer Studium Mathematik -- Master},
  Series = {Springer Stud. Math. -- Master},
  ISSN = {2509-9310},
  ISBN = {978-3-658-43030-6; 978-3-658-43031-3},
  Year = {2023},
  Publisher = {Wiesbaden: Springer Spektrum},
  Language = {English},
  DOI = {10.1007/978-3-658-43031-3},
  zbMATH = {7802900}
}

@Article{Hall/Rydh:2017,
  Author = {Hall, Jack and Rydh, David},
  Title = {Perfect complexes on algebraic stacks},
  FJournal = {Compositio Mathematica},
  Journal = {Compos. Math.},
  ISSN = {0010-437X},
  Volume = {153},
  Number = {11},
  Pages = {2318--2367},
  Year = {2017},
  Language = {English},
  DOI = {10.1112/S0010437X17007394},
  zbMATH = {6810455},
  Zbl = {1390.14057}
}

@misc{DeDeyn/Lank/ManaliRahul:2024b,
  title={Descent and generation for noncommutative coherent algebras over schemes}, 
  author={Timothy {De Deyn} and Pat Lank and Kabeer {Manali Rahul}},
  year={2024},
  url={https://arxiv.org/abs/2410.01785}, 
  eprint={2410.01785},
  archivePrefix={arXiv},
  howpublished	= {\href{https://arxiv.org/abs/2410.01785}{arxiv:2410.01785}},
  publisher     = {arXiv},
}

@article{Neeman:2023b,
  author = {Neeman, Amnon},
  title = {An improvement on the base-change theorem and the functor {{\(f^!\)}}},
  fjournal = {Bulletin of the Iranian Mathematical Society},
  journal = {Bull. Iran. Math. Soc.},
  issn = {1017-060X},
  volume = {49},
  number = {3},
  pages = {163},
  note = {Id/No 25},
  year = {2023},
  language = {English},
  doi = {10.1007/s41980-023-00768-6},
  zbMATH = {7700381},
  Zbl = {1527.14038}
}

@misc{DeDeyn/Lank/ManaliRahul:2025,
  title={Descending strong generation in algebraic geometry},
  author={Timothy {De Deyn} and Pat Lank and Kabeer {Manali Rahul}},
  year={2025},
  url={https://arxiv.org/abs/2502.08629},
  eprint={2502.08629},
  archivePrefix={arXiv},
  howpublished    = {\href{https://arxiv.org/abs/2502.08629}{arXiv:2502.08629}},
  publisher     = {arXiv},
}

@Article{Elagin/Lunts/Schnurer:2020,
  Author = {Elagin, Alexey and Lunts, Valery A. and Schn{\"u}rer, Olaf M.},
  Title = {Smoothness of derived categories of algebras},
  FJournal = {Moscow Mathematical Journal},
  Journal = {Mosc. Math. J.},
  ISSN = {1609-3321},
  Volume = {20},
  Number = {2},
  Pages = {277--309},
  Year = {2020},
  Language = {English},
  URL = {www.mathjournals.org/mmj/2020-020-002/2020-020-002-003.html},
  zbMATH = {7206644},
  Zbl = {1468.16021}
}

@Article{Bergh/Lunts/Schnurer:2016,
  Author = {Bergh, Daniel and Lunts, Valery A. and Schn{\"u}rer, Olaf M.},
  Title = {Geometricity for derived categories of algebraic stacks},
  FJournal = {Selecta Mathematica. New Series},
  Journal = {Sel. Math., New Ser.},
  ISSN = {1022-1824},
  Volume = {22},
  Number = {4},
  Pages = {2535--2568},
  Year = {2016},
  Language = {English},
  DOI = {10.1007/s00029-016-0280-8},
  zbMATH = {6668639},
  Zbl = {1360.14058}
}

@Article{Hall:2023,
  Author = {Hall, Jack},
  Title = {{GAGA} theorems},
  FJournal = {Journal de Math{\'e}matiques Pures et Appliqu{\'e}es. Neuvi{\`e}me S{\'e}rie},
  Journal = {J. Math. Pures Appl. (9)},
  ISSN = {0021-7824},
  Volume = {175},
  Pages = {109--142},
  Year = {2023},
  Language = {English},
  DOI = {10.1016/j.matpur.2023.05.004},
  zbMATH = {7693671},
  Zbl = {1527.14033}
}

@Article{Rouquier:2008,
  Author = {Rouquier, Rapha{\"e}l},
  Title = {Dimensions of triangulated categories.},
  FJournal = {Journal of \(K\)-Theory},
  Journal = {J. \(K\)-Theory},
  ISSN = {1865-2433},
  Volume = {1},
  Number = {2},
  Pages = {193--256},
  Year = {2008},
  Language = {English},
  zbMATH = {5348081},
  Zbl = {1165.18008}
}

@Article{Iyengar/Takahashi:2019,
  Author = {Iyengar, Srikanth B. and Takahashi, Ryo},
  Title = {Openness of the regular locus and generators for module categories},
  FJournal = {Acta Mathematica Vietnamica},
  Journal = {Acta Math. Vietnam.},
  ISSN = {0251-4184},
  Volume = {44},
  Number = {1},
  Pages = {207--212},
  Year = {2019},
  Language = {English},
  DOI = {10.1007/s40306-018-0294-8},
  zbMATH = {7075132},
  Zbl = {1420.13038}
}

@article{Dey/Lank:2024a,
  author = {Dey, Souvik and Lank, Pat},
  title = {Closedness of the singular locus and generation for derived categories},
  fjournal = {Journal of Algebra},
  journal = {J. Algebra},
  issn = {0021-8693},
  volume = {684},
  pages = {64--77},
  year = {2025},
  language = {English},
  doi = {10.1016/j.jalgebra.2025.07.007},
  zbMATH = {8098410}
}

@Book{Beilinson/Berstein/Deligne/Gabber:2018,
  Author = {Beilinson, Alexander and Bernstein, Joseph and Deligne, Pierre and Gabber, Ofer},
  Title = {Faisceaux pervers. {Actes} du colloque ``{Analyse} et {Topologie} sur les {Espaces} {Singuliers}''. {Partie} {I}},
  Edition = {2nd edition},
  FSeries = {Ast{\'e}risque},
  Series = {Ast{\'e}risque},
  ISSN = {0303-1179},
  Volume = {100},
  ISBN = {978-2-85629-878-7},
  Year = {2018},
  Publisher = {Paris: Soci{\'e}t{\'e} Math{\'e}matique de France (SMF)},
  Language = {French},
  zbMATH = {6868966},
  Zbl = {1390.14055}
}

@misc{Canonaco/Haesemeyer/Neeman/Stellari:2024,
  title={The passage among the subcategories of weakly approximable triangulated categories},
  author={Alberto Canonaco and Christian Haesemeyer and Amnon Neeman and Paolo Stellari},
  year={2024},
  url={https://arxiv.org/abs/2402.04605},
  eprint={2402.04605},
  archivePrefix={arXiv},
  howpublished    = {\href{https://arxiv.org/abs/2402.04605}{arXiv:2402.04605}},
  publisher     = {arXiv},
}

@article {Bondal/VandenBergh:2003,
  AUTHOR = {Bondal, Alexei and {V}an den {B}ergh, Michel},
  TITLE = {Generators and representability of functors in commutative and
          noncommutative geometry},
  JOURNAL = {Mosc. Math. J.},
  FJOURNAL = {Moscow Mathematical Journal},
  VOLUME = {3},
  YEAR = {2003},
  NUMBER = {1},
  PAGES = {1--36, 258},
  ISSN = {1609-3321,1609-4514},
  MRCLASS = {18E30 (14F05)},
  MRNUMBER = {1996800},
  MRREVIEWER = {Ioannis\ Emmanouil},
    DOI = {10.17323/1609-4514-2003-3-1-1-36},
    URL = {https://doi.org/10.17323/1609-4514-2003-3-1-1-36},
}

@Article{Neeman:2021a,
  Author = {Neeman, Amnon},
  Title = {Strong generators in {{\(\mathbf{D}^{\mathrm{perf}}(X)\)}} and {{\(\mathbf{D}^b_{\mathrm{coh}}(X)\)}}},
  FJournal = {Annals of Mathematics. Second Series},
  Journal = {Ann. Math. (2)},
  ISSN = {0003-486X},
  Volume = {193},
  Number = {3},
  Pages = {689--732},
  Year = {2021},
  Language = {English},
  DOI = {10.4007/annals.2021.193.3.1},
  zbMATH = {7353240},
  Zbl = {1478.18014}
}

@Article{Letz:2021,
  Author = {Letz, Janina C.},
  Title = {Local to global principles for generation time over commutative {Noetherian} rings},
  FJournal = {Homology, Homotopy and Applications},
  Journal = {Homology Homotopy Appl.},
  ISSN = {1532-0073},
  Volume = {23},
  Number = {2},
  Pages = {165--182},
  Year = {2021},
  Language = {English},
  DOI = {10.4310/HHA.2021.v23.n2.a10},
  zbMATH = {7420249},
  Zbl = {1485.18021}
}

@Article{Letz:2020,
  Author={Letz, Janina C.},
  Year={2020},
  Title={Generation Time in Derived Categories},
  ISBN = {9798841777540},
  Journal={ProQuest Dissertations and Theses}
}

@article {Lank/Olander:2025,
  AUTHOR = {Lank, Pat and Olander, Noah},
  TITLE = {Approximation by perfect complexes detects {R}ouquier
          dimension},
  JOURNAL = {Mosc. Math. J.},
  FJOURNAL = {Moscow Mathematical Journal},
  VOLUME = {25},
  YEAR = {2025},
  NUMBER = {1},
  PAGES = {13--31},
  ISSN = {1609-3321,1609-4514},
  MRCLASS = {14F08 (14A30 14B05 18G80)},
  MRNUMBER = {4885531},
}

@Article{Hall/Rydh:2018,
  Author = {Hall, Jack and Rydh, David},
  Title = {Addendum to: ``{\'E}tale d{\'e}vissage, descent and pushouts of stacks''},
  FJournal = {Journal of Algebra},
  Journal = {J. Algebra},
  ISSN = {0021-8693},
  Volume = {498},
  Pages = {398--412},
  Year = {2018},
  Language = {English},
  DOI = {10.1016/j.jalgebra.2017.11.027},
  zbMATH = {6834838},
  Zbl = {1441.14006}
}

@Article{AlonsoTarrio/LopezJeremias/Salorio:2003,
  Author = {Alonso Tarr{\'{\i}}o, Leovigildo and L{\'o}pez, Ana Jerem{\'{\i}}as and Salorio, Mar{\'{\i}}a Jos{\'e} Souto},
  Title = {Construction of {{\(t\)}}-structures and equivalences of derived categories},
  FJournal = {Transactions of the American Mathematical Society},
  Journal = {Trans. Am. Math. Soc.},
  ISSN = {0002-9947},
  Volume = {355},
  Number = {6},
  Pages = {2523--2543},
  Year = {2003},
  Language = {English},
  DOI = {10.1090/S0002-9947-03-03261-6},
  zbMATH = {1896867},
  Zbl = {1019.18007}
}

@Article{Keller/Vossieck:1988,
  Author = {Keller, B. and Vossieck, D.},
  Title = {Aisles in derived categories},
  FJournal = {Bulletin de la Soci{\'e}t{\'e} Math{\'e}matique de Belgique. S{\'e}rie A},
  Journal = {Bull. Soc. Math. Belg., S{\'e}r. A},
  ISSN = {0037-9476},
  Volume = {40},
  Number = {2},
  Pages = {239--253},
  Year = {1988},
  Language = {English},
  zbMATH = {4097609},
  Zbl = {0671.18003}
}

@Article{AlonsoTarrio/JeremiasLopez/Salorio/Souto:2003,
  Author = {Alonso Tarr{\'{\i}}o, Leovigildo and L{\'o}pez, Ana Jerem{\'{\i}}as and Salorio, Mar{\'{\i}}a Jos{\'e} Souto},
  Title = {Construction of {{\(t\)}}-structures and equivalences of derived categories},
  FJournal = {Transactions of the American Mathematical Society},
  Journal = {Trans. Am. Math. Soc.},
  ISSN = {0002-9947},
  Volume = {355},
  Number = {6},
  Pages = {2523--2543},
  Year = {2003},
  Language = {English},
  DOI = {10.1090/S0002-9947-03-03261-6},
  zbMATH = {1896867},
  Zbl = {1019.18007}
}

@misc{Lank:2026,
  title={Thomason condition for regular algebraic stacks}, 
  author={Pat Lank},
  year={2026},
  url={https://arxiv.org/abs/2601.04053}, 
  eprint={2601.04053},
  archivePrefix={arXiv},
  howpublished	= {\href{https://arxiv.org/abs/2601.04053}{arXiv:2601.04053}},
  publisher     = {arXiv},
}

@article {Neeman:2022,
  AUTHOR = {Neeman, Amnon},
  TITLE = {Bounded {$t$}-structures on the category of perfect complexes},
  JOURNAL = {Acta Math.},
  FJOURNAL = {Acta Mathematica},
  VOLUME = {233},
  YEAR = {2024},
  NUMBER = {2},
  PAGES = {239--284},
  ISSN = {0001-5962,1871-2509},
  MRCLASS = {18G80 (14F08 18N55 19D35)},
  MRNUMBER = {4827655},
  DOI = {10.4310/acta.2024.v233.n2.a2},
  URL = {https://doi.org/10.4310/acta.2024.v233.n2.a2},
}

@misc{Rydh:2023,
  title={Absolute noetherian approximation of algebraic stacks},
  author={David Rydh},
  year={2023},
  url={https://arxiv.org/abs/arXiv:2311.09208},
  eprint={2311.09208},
  archivePrefix={arXiv},
  howpublished    = {\href{https://arxiv.org/abs/2311.09208}{arxiv:2311.09208}},
  publisher     = {arXiv},
}

@Article{Rydh:2011,
  Author = {Rydh, David},
  Title = {{\'E}tale d{\'e}vissage, descent and pushouts of stacks},
  FJournal = {Journal of Algebra},
  Journal = {J. Algebra},
  ISSN = {0021-8693},
  Volume = {331},
  Number = {1},
  Pages = {194--223},
  Year = {2011},
  Language = {English},
  DOI = {10.1016/j.jalgebra.2011.01.006},
  zbMATH = {5949289},
  Zbl = {1230.14005}
}

@Article{Avramov/Iyengar/Lipman:2010,
  Author = {Avramov, Luchezar L. and Iyengar, Srikanth B. and Lipman, Joseph},
  Title = {Reflexivity and rigidity for complexes, {I}: {Commutative} rings},
  FJournal = {Algebra \& Number Theory},
  Journal = {Algebra Number Theory},
  ISSN = {1937-0652},
  Volume = {4},
  Number = {1},
  Pages = {47--86},
  Year = {2010},
  Language = {English},
  DOI = {10.2140/ant.2010.4.47},
  zbMATH = {5704462},
  Zbl = {1194.13017}
}

@Article{Burke/Neeman/Pauwels:2023,
  Author = {Burke, Jesse and Neeman, Amnon and Pauwels, Bregje},
  Title = {Gluing approximable triangulated categories},
  FJournal = {Forum of Mathematics, Sigma},
  Journal = {Forum Math. Sigma},
  ISSN = {2050-5094},
  Volume = {11},
  Pages = {18},
  Note = {Id/No e110},
  Year = {2023},
  Language = {English},
  DOI = {10.1017/fms.2023.97},
  zbMATH = {7781650},
  Zbl = {1536.18009}
}

@Article{Hall:2016,
  Author = {Hall, Jack},
  Title = {The {Balmer} spectrum of a tame stack},
  FJournal = {Annals of \(K\)-Theory},
  Journal = {Ann. \(K\)-Theory},
  ISSN = {2379-1683},
  Volume = {1},
  Number = {3},
  Pages = {259--274},
  Year = {2016},
  Language = {English},
  DOI = {10.2140/akt.2016.1.259},
  zbMATH = {6606694},
  Zbl = {1374.14016}
}

@InCollection{Neeman:2023,
  Author = {Neeman, Amnon},
  Title = {Finite approximations as a tool for studying triangulated categories},
  BookTitle = {International congress of mathematicians 2022, ICM 2022, Helsinki, Finland, virtual, July 6--14, 2022. Volume 3. Sections 1--4},
  ISBN = {978-3-98547-061-7; 978-3-98547-561-2; 978-3-98547-058-7; 978-3-98547-558-2},
  Pages = {1636--1658},
  Year = {2023},
  Publisher = {Berlin: European Mathematical Society (EMS)},
  Language = {English},
  DOI = {10.4171/ICM2022/35},
  zbMATH = {7823037}
}

@article {Stevenson:2025,
  AUTHOR = {Stevenson, Greg},
  TITLE = {Rouquier dimension versus global dimension},
  JOURNAL = {J. Pure Appl. Algebra},
  FJOURNAL = {Journal of Pure and Applied Algebra},
  VOLUME = {229},
  YEAR = {2025},
  NUMBER = {1},
  PAGES = {Paper No. 107827, 4},
  ISSN = {0022-4049,1873-1376},
  MRCLASS = {18G20 (13C60)},
  MRNUMBER = {4811077},
  DOI = {10.1016/j.jpaa.2024.107827},
  URL = {https://doi.org/10.1016/j.jpaa.2024.107827},
}

@Article{Hall/Rydh:2015,
  Author = {Hall, Jack and Rydh, David},
  Title = {Algebraic groups and compact generation of their derived categories of representations},
  FJournal = {Indiana University Mathematics Journal},
  Journal = {Indiana Univ. Math. J.},
  ISSN = {0022-2518},
  Volume = {64},
  Number = {6},
  Pages = {1903--1923},
  Year = {2015},
  Language = {English},
  DOI = {10.1512/iumj.2015.64.5719},
  URL = {www.iumj.indiana.edu/IUMJ/ABS/2015/5719},
  zbMATH = {6543228},
  Zbl = {1348.14045}
}

@article {Hall/Rydh:2023,
  AUTHOR = {Hall, Jack and Rydh, David},
  TITLE = {Mayer-{V}ietoris squares in algebraic geometry},
  JOURNAL = {J. Lond. Math. Soc. (2)},
  FJOURNAL = {Journal of the London Mathematical Society. Second Series},
  VOLUME = {107},
  YEAR = {2023},
  NUMBER = {5},
  PAGES = {1583--1612},
  ISSN = {0024-6107,1469-7750},
  MRCLASS = {14A20 (13F40 14F08 14F20)},
  MRNUMBER = {4585296},
  MRREVIEWER = {Ariyan\ Javanpeykar},
  DOI = {10.1112/jlms.12575},
  URL = {https://doi.org/10.1112/jlms.12575},
}

@article {Hall/Neeman/Rydh:2019,
  AUTHOR = {Hall, Jack and Neeman, Amnon and Rydh, David},
  TITLE = {One positive and two negative results for derived categories
        of algebraic stacks},
  JOURNAL = {J. Inst. Math. Jussieu},
  FJOURNAL = {Journal of the Institute of Mathematics of Jussieu. JIMJ.
        Journal de l'Institut de Math\'ematiques de Jussieu},
  VOLUME = {18},
  YEAR = {2019},
  NUMBER = {5},
  PAGES = {1087--1111},
  ISSN = {1474-7480,1475-3030},
  MRCLASS = {14F08 (14A20 14G17 18G10)},
  MRNUMBER = {3995721},
  MRREVIEWER = {Ariyan\ Javanpeykar},
  DOI = {10.1017/s1474748017000366},
  URL = {https://doi.org/10.1017/s1474748017000366},
}

@article {Antieau/Gepner/Heller:2019,
  AUTHOR = {Antieau, Benjamin and Gepner, David and Heller, Jeremiah},
  TITLE = {{$K$}-theoretic obstructions to bounded {$t$}-structures},
  JOURNAL = {Invent. Math.},
  FJOURNAL = {Inventiones Mathematicae},
  VOLUME = {216},
  YEAR = {2019},
  NUMBER = {1},
  PAGES = {241--300},
  ISSN = {0020-9910,1432-1297},
  MRCLASS = {19D35 (16E45 16P40 18E10 18E30 55P43)},
  MRNUMBER = {3935042},
  MRREVIEWER = {Guillermo\ Corti\~nas},
  DOI = {10.1007/s00222-018-00847-0},
  URL = {https://doi.org/10.1007/s00222-018-00847-0},
}

@book {Olsson:2016,
  AUTHOR = {Olsson, Martin},
  TITLE = {Algebraic spaces and stacks},
  SERIES = {American Mathematical Society Colloquium Publications},
  VOLUME = {62},
  PUBLISHER = {American Mathematical Society, Providence, RI},
  YEAR = {2016},
  PAGES = {xi+298},
  ISBN = {978-1-4704-2798-6},
  MRCLASS = {14D23 (14D22)},
  MRNUMBER = {3495343},
  MRREVIEWER = {Stefan\ Schr\"oer},
  DOI = {10.1090/coll/062},
  URL = {https://doi.org/10.1090/coll/062},
  }

@article {Alper:2013,
  AUTHOR = {Alper, Jarod},
  TITLE = {Good moduli spaces for {A}rtin stacks},
  JOURNAL = {Ann. Inst. Fourier (Grenoble)},
  FJOURNAL = {Universit\'e{} de Grenoble. Annales de l'Institut Fourier},
  VOLUME = {63},
  YEAR = {2013},
  NUMBER = {6},
  PAGES = {2349--2402},
  ISSN = {0373-0956,1777-5310},
  MRCLASS = {14D23 (14L24 14L30)},
  MRNUMBER = {3237451},
  MRREVIEWER = {Arvid\ Siqveland},
  DOI = {10.5802/aif.2833},
  URL = {https://doi.org/10.5802/aif.2833},
}

@article {Rydh:2016,
  author = {Rydh, David},
  title = {Approximation of Sheaves on Algebraic Stacks},
  journal = {International Mathematics Research Notices},
  volume = {2016},
  number = {3},
  pages = {717-737},
  year = {2015},
  month = {05},
  issn = {1073-7928},
  doi = {10.1093/imrn/rnv142},
  url = {https://doi.org/10.1093/imrn/rnv142},
  eprint = {https://academic.oup.com/imrn/article-pdf/2016/3/717/7374761/rnv142.pdf},
}

@article {Olsson:2007,
  AUTHOR = {Olsson, Martin},
  TITLE = {Sheaves on {A}rtin stacks},
  JOURNAL = {J. Reine Angew. Math.},
  FJOURNAL = {Journal f\"ur die Reine und Angewandte Mathematik. [Crelle's
        Journal]},
  VOLUME = {603},
  YEAR = {2007},
  PAGES = {55--112},
  ISSN = {0075-4102,1435-5345},
  MRCLASS = {14A20 (14D20)},
  MRNUMBER = {2312554},
  MRREVIEWER = {Charles\ D.\ Cadman},
  DOI = {10.1515/CRELLE.2007.012},
  URL = {https://doi.org/10.1515/CRELLE.2007.012},
}

@article {Laszlo/Olsson:2008,
  AUTHOR = {Laszlo, Yves and Olsson, Martin},
  TITLE = {The six operations for sheaves on {A}rtin stacks. {I}.
  {F}inite coefficients},
  JOURNAL = {Publ. Math. Inst. Hautes \'Etudes Sci.},
  FJOURNAL = {Publications Math\'ematiques. Institut de Hautes \'Etudes
  Scientifiques},
  NUMBER = {107},
  YEAR = {2008},
  PAGES = {109--168},
  ISSN = {0073-8301,1618-1913},
  MRCLASS = {14A20 (14F05)},
  MRNUMBER = {2434692},
  MRREVIEWER = {Charles\ D.\ Cadman},
  DOI = {10.1007/s10240-008-0011-6},
  URL = {https://doi.org/10.1007/s10240-008-0011-6},
}

@article {Jatoba:2021,
    AUTHOR = {Jatoba, V. B.},
     TITLE = {Strong generators in {\(\mathbf{D}^{\mathrm{perf}}(X)\)} for
              schemes with a separator},
   JOURNAL = {Proc. Amer. Math. Soc.},
  FJOURNAL = {Proceedings of the American Mathematical Society},
    VOLUME = {149},
      YEAR = {2021},
    NUMBER = {5},
     PAGES = {1957--1971},
      ISSN = {0002-9939,1088-6826},
   MRCLASS = {18G80 (14F08 18G20)},
  MRNUMBER = {4232189},
MRREVIEWER = {Umesh\ V.\ Dubey},
       DOI = {10.1090/proc/15353},
       URL = {https://doi.org/10.1090/proc/15353},
}

\end{document}